\documentclass[11pt, a4paper, reqno]{amsart}

\usepackage[foot]{amsaddr}
\makeatletter
\renewcommand{\email}[2][]{%
	\ifx\emails\@empty\relax\else{\g@addto@macro\emails{,\space}}\fi%
	\@ifnotempty{#1}{\g@addto@macro\emails{\textrm{(#1)}\space}}%
	\g@addto@macro\emails{#2}%
}
\makeatother

\usepackage{amsthm,amsmath, amsfonts, amssymb, amscd,enumerate,enumitem, amscd, xcolor, url, mathrsfs, esint, cancel}
\usepackage[colorlinks]{hyperref}
\usepackage[utf8]{inputenc}
\usepackage{mathtools}
\usepackage{soul}

\usepackage[margin=2cm, marginparwidth=1.5cm]{geometry} 

\newtheorem{thm}{Theorem}[section]

\newtheorem{lem}[thm]{Lemma}
\newtheorem{prop}[thm]{Proposition}
\theoremstyle{definition}
\newtheorem{defn}[thm]{Definition}
\theoremstyle{remark}
\newtheorem{rem}[thm]{Remark}

\numberwithin{equation}{section}

\newcommand{\BMO}{\textup{BMO}_\rho}
\newcommand{\loc}{\textup{loc}}

\DeclareMathOperator{\supp}{supp}

\allowdisplaybreaks[3]
\setlist[enumerate,1]{label=(\roman*)}

\begin{document}
	
	\title[Weighted estimates for Schr\"odinger--Calder\'on--Zygmund operators with exponential...]
	{Weighted estimates for Schr\"odinger--Calder\'on--Zygmund operators with exponential decay}
	
\author[E. Dalmasso]{Estefan\'ia Dalmasso$^{1}$}
\address{$^1$Instituto de Matem\'atica Aplicada del Litoral, UNL, CONICET, FIQ.\newline Colectora Ruta Nac. Nº 168, Paraje El Pozo,\newline S3007ABA, Santa Fe, Argentina}
\email
{edalmasso@santafe-conicet.gov.ar}

\author[G. R. Lezama]{Gabriela R. Lezama$^{1}$}
\email
{rlezama@santafe-conicet.gov.ar}

\author[M. Toschi]{Marisa Toschi$^{2}$}
\address{$^2$Instituto de Matem\'atica Aplicada del Litoral, UNL, CONICET, FHUC.\newline Colectora Ruta Nac. Nº 168, Paraje El Pozo,\newline S3007ABA, Santa Fe, Argentina}
\email{mtoschi@santafe-conicet.gov.ar}

	\date{\today}
	\subjclass{Primary 42B20; Secondary 35J10, 42B25, 42B35, 47G10} 
	
	\keywords{Extrapolation, Schr\"odinger operators, Riesz transforms, Laplace multipliers, weights}

	\begin{abstract}
		In this work we obtain weighted boundedness results for singular integral operators with kernels exhibiting exponential decay. We also show that the classes of weights are characterized by a suitable maximal operator. Additionally, we study the boundedness of various operators associated with the generalized Schr\"odinger operator $-\Delta + \mu$, where $\mu$ is a nonnegative Radon measure in $\mathbb{R}^d$, for~${d\geq 3}$.
	\end{abstract}
	\maketitle
	
\section{Introduction and main result}

In \cite{Bailey21}, J. Bailey studied about harmonic analysis related to the Schr\"odinger operator $\mathcal{L}= -\Delta+V$, for $V$ a locally integrable function in $\mathbb{R}^d$ with $d\geq 3$. 
Asking the question  whether
it is possible to construct a Muckenhoupt-type class of weights adapted to the underlying differential
operator and whether this class can be characterized in terms of the corresponding
operators, he defined new classes of weights. These classes, named  $S^V_{p,c}$ and $H^{V,m}_{p,c}$, are comparable under certain dependence on the parameters $m$ and $c$. The second one looks at the weights in the Euclidean distance, while the first one is adapted to another distance depending on the potential $V$. Both classes of weights are 
strictly larger than the class $A_p^{V, \infty}$ given in \cite{BHS11} (see \S\ref{sec: weights}).

Actually, since the classes can be described through a critical radius function associated to the operator~$\mathcal{L}$, we will name the classes  $S^\rho_{p,c}$, $H^{\rho,m}_{p,c}$ and  $A^\rho_{p}$, respectively. This idea together with the main tool used throughout the work, the decay of the fundamental solution of the operator $\mathcal{L}$, leads us to think about leaving particular examples aside, for the moment, and address a general theory    to prove the boundedness of singular integral operators having kernels with exponential decay on weighted $L^p$ spaces. Regarding the weights, we will deal with the class {$H^{\rho,m}_{p,c}$} since one can have a better understanding of the behavior of the weights with respect to the Euclidean distance. An analysis of the classes of weights is given in \S\ref{subsec: weights properties}, including a simple example that shows the difference between them.

In the spirit of \cite[Theorem~5.1]{BCH13Extrapol}, we give an extrapolation theorem (see Theorem~\ref{thm: extrapolation H weights}) on weighted Lebesgue spaces for weights associated to the maximal operator defined in \cite{Bailey21}, and prove boundedness results for certain collection of operators that we shall introduce in \S\ref{sec: operators}. In order to do so, we first complete the work done by J. Bailey and prove that the classes {$H^{\rho,m}_{p,c}$} are sufficient to ensure the boundedness of the mentioned maximal function, which is an important result by itself (see \S\ref{subsec: maximal operators}).

The article will conclude with the presentation of certain examples of operators associated with the differential operator $\mathcal{L}_\mu = -\Delta + \mu$, along with their corresponding boundedness properties. Here, $\mu$ is a nonnegative Radon measure on \(\mathbb{R}^d\), as considered in \cite{Shen99} and \cite{Bailey21}. These operators include the Riesz transforms $R_\mu=\nabla (-\Delta + \mu)^{-\frac{1}{2}}$ and its adjoint ${R^*_\mu=(-\Delta + \mu)^{-\frac{1}{2}} \nabla }$, Laplace transform type multipliers, as well as operators of the form $(-\Delta +V)^{-j/2} V^{j/2}$ for $j=1,2$ with $V$ a   nonnegative function belonging to the classical reverse H\"older class $RH_q$ for $q>\frac d2.$ 

Throughout this paper $C$ and $c$ will always denote positive constants than may change in each occurrence. Also, by $a\lesssim b$ we mean there exists a positive constant $c$ such that $a\leq cb$, and $a\gtrsim b$ is the same as $b\lesssim a$. When both are valid, we simply write $a\sim b$.

\section{Preliminaries}

\subsection{Critical radius function and Agmon distance}
We introduce the notion of critical radius function as any function satisfying the definition given below.
\begin{defn}\label{def: critical radius function}
	A function $\rho: \mathbb R^d\to [0,\infty)$ will be called  a \textbf{critical radius function} if there exist constants $k_0, C_0 \geq 1$ such that
	\begin{equation}\label{eq: critical radius ineq}
		C_0^{-1}\rho(x)\left(1 + \frac{|x-y|}{\rho(x)}\right)^{-k_0}\leq \rho(y)\leq C_0\rho(x)\left(1+\frac{|x-y|}{\rho(x)}\right)^{\frac{k_0}{k_0+1}}
	\end{equation}
	for $x,y \in \mathbb{R}^d$. 
\end{defn}

\begin{rem}
	It follows from the definition that $\rho(x)\sim \rho(y)$ if  $|x-y|\lesssim \rho(x)$.
\end{rem}

A ball of the form $B(x, \rho(x))$, with $x \in \mathbb{R}^d$, is called \textbf{critical} and a ball $B(x,r)$ with $r\leq  \rho(x)$ will be called \textbf{sub-critical}. We denote by $\mathcal{B}_\rho$ the family of all sub-critical balls, that is,
\begin{equation*}
	\mathcal{B}_\rho=\left\{B(x,r): x \in \mathbb{R}^d,\,  r\leq \rho(x)\right\}.
\end{equation*}

It is well-known (see, for instance, \cite{Agmon} and \cite{Helffer}) that many phenomena involving the classical Schr\"odinger operator $L_V=-\Delta+V$, with a potential $V$, can be better described in terms of the Agmon distance, which is comparable to the Euclidean distance locally, but is a Riemannian-like distance obtained by deforming the Euclidean one with respect to a quadratic form involving the potential. 

In our setting, a notion of Agmon distance can also be given as:
\begin{equation*}
	d_\rho(x,y):= \inf_\gamma \int_{0}^{1} \rho(\gamma(t))^{-1}|\gamma'(t)|dt,
\end{equation*} 
where the infimum is taken over every absolutely continuous function   ${\gamma:[0,1]\rightarrow \mathbb{R}^d}$ with  $\gamma(0)=x$ and $\gamma(1)=y$. 

The following lemma establishes the aforementioned local comparison between the Agmon distance and the Euclidean distance.

\begin{lem}[{\cite[Lemma~2.2]{Bailey21}}]\label{eq: drho equiv usual dist}
	Let $\rho:\mathbb{R} \rightarrow [0,\infty)$ be a critical radius function. If $|x-y| \leq 2\rho(x)$, there exists $D_{0}>1$ such that 
	\begin{equation}
		\frac{1}{D_0}\frac{|x-y|}{\rho(x)}\leq d_\rho(x,y)\leq D_{0}\frac{|x-y|}{\rho(x)}.
	\end{equation}
\end{lem}

The Agmon distance can also be compared with the quantity  $1+\frac{|x-y|}{\rho(x)}$ as stated below.

\begin{lem}[{\cite[Lemma~2.3]{Bailey21}, \cite[(3.19)]{Shen99}}]\label{lem: drho comparable local}
	Let $\rho:\mathbb{R} \rightarrow [0,\infty)$ be a critical radius function. Then, 
	\begin{equation*}
		d_\rho(x,y) \leq C_0 \left(1+ \frac{| x-y| }{\rho(x)}\right)^{k_0+1}  
	\end{equation*}
	for every $x,y \in \mathbb{R}^d$, where $C_0$ and $k_0$ are the constants appearing in \eqref{eq: critical radius ineq}.
	
	Moreover, when  $|x-y|\geq \rho(x)$, there exists a constant $D_{1}>1$ such that
	\begin{equation*}
		d_\rho(x,y) \geq \frac{1}{D_1}\left(1+\frac{|x-y| }{\rho(x)}\right)^{\frac{1}{k_0+1}}.
	\end{equation*}
\end{lem}

\subsection{Spaces of functions}

Given a weight $w$, i.e., a nonnegative and locally integrable function defined on $\mathbb{R}^d$, for $1\leq p<\infty$ we say that a measurable function $f$ on $\mathbb{R}^d$ belongs to $L^p(w)$ iff $\|f\|_{L^p(w)}:=\left(\int_{\mathbb{R}^d}|f(x)|^p w(x)dx\right)^{1/p}<\infty$. 
For $w\equiv 1$ we simply write $L^p(w)=L^p$. 

When $p=\infty$, $f\in L^\infty(w)$ iff $fw\in L^\infty$ and we write $\|f\|_{L^{\infty}(w)}=\|fw\|_\infty$ where this quantity denotes, as usual, the essential supremum of $fw$ on $\mathbb{R}^d$.

Given a critical radius function~$\rho$ and a weight $w$, we say that a locally integrable function $f$ belongs to $\BMO(w)$ if it satisfies
\begin{equation}\label{eq: osc average bmo}
	\frac{\|\chi_B w\|_\infty}{|B|}\int_B| f-f_B| \leq C, \quad \text{for every ball } B=B(x,r) \text{ with }r<\rho(x)
\end{equation}  
and 
\begin{equation}\label{eq: average bmo}
	\frac{\|\chi_B w\|_\infty}{|B|}\int_B | f | \leq C, \quad  \text{for every ball }B=B(x,r)\text{ with } r= \rho(x),
\end{equation}
where $f_B$ stands for the average of $f$ over the ball $B$. As usual, the norm $\|f\|_{\BMO(w)}$ is defined as the maximum of the infima constants appearing in \eqref{eq: osc average bmo} and \eqref{eq: average bmo}.

It is a classical fact that
\[\sup_{B}\frac{1}{|B|}\int_B| f-f_B|\sim \sup_{B} \inf_{a\in \mathbb{R}}\frac{1}{|B|}\int_B| f-a|.\]
This allows us to verify \eqref{eq: osc average bmo} for some constant $a\in \mathbb R$, not necessarily the average~$f_B$.

Moreover, as it was proved in \cite{BCH13Extrapol}, the space $\BMO(w)$ can be characterized in terms of a suitable sharp maximal function.

\begin{lem}[{\cite[Lemma~2]{BCH13Extrapol}}]\label{lem: BMO Msharp}
	For any weight $w$ and any $f\in L^1_{\loc}(\mathbb{R}^d)$,   $\|f\|_{\BMO(w)}\sim \|M_{\loc}^\sharp(f)w\|_{\infty}$
	where
	\[M_{\loc}^\sharp(f)(x):=\sup_{x\in B\in \mathcal{B}_{\rho}}\frac{1}{|B|}\int_B|f-f_B|+\sup_{x\in B=B(y,\rho(y))}\frac{1}{|B|}\int_B|f|.\]
\end{lem}

\subsection{Schr\"odinger--Calder\'on--Zygmund operators with exponential decay} \label{sec: operators}

In the following results we consider an operator $T$ with kernel $K$, understood in the principal value sense, that satisfies certain Calder\'on--Zygmund or H\"ormander-type conditions, having exponential decay. The terminology used here follows \cite{BHQ19} (see also \cite{MSTZ}), although their operators, characterized by polynomial decay, were previously considered in \cite{BCH13Extrapol} without being explicitly named.

\begin{defn}\label{def: s-delta type}
	For $1<s< \infty$ and $0<\delta \leq 1$ we say that $T$ is  an \textbf{exponential Schr\"odinger--Calder\'on--Zygmund operator of $(s, \delta)$ type} if
	\begin{enumerate}
		\item \label{itm: weak-type T}  $T$ is bounded from $L^{s'}(\mathbb{R}^d)$ to $L^{s',\infty}(\mathbb{R}^d)$ (where $s'$ denotes the H\"older conjugate of $s$); 
		\item $T $ has an associated kernel $K$ verifying the following conditions:
		\begin{enumerate}
			\item there exist constants $c,C >0$ and $m>0$ such that 
			\begin{equation}\label{eq: size Hormander}
				\left(\frac{1}{R^d}\int_{R<|x_0-y|\leq 2R}|K(x,y)|^s dy \right)^{1/s}
				\leq 
				\frac{C}{R^{d}}\exp\left(-c \left(1+\frac{R}{\rho(x)}\right)^{m}\right),
			\end{equation}
			whenever $|x-x_0|<R/2$;
			\item there exist a constant $C>0$ such that
			\begin{equation}\label{eq: smooth Hormander}
				\left(\frac{1}{R^d}\int_{R<|x_0-y|\leq 2R}|K(x,y)-K(x_0,y)|^s dy \right)^{1/s}
				\leq  
				\frac{C}{R^{d}}\left(\frac{r}{R}\right)^\delta 
			\end{equation}
			for every $|x-x_0|<r\leq \rho(x_0)$ and $r<R/2$.
		\end{enumerate}
	\end{enumerate}
\end{defn}

For the case $s = \infty$, we consider pointwise estimates instead, for which we introduce the following definition.

\begin{defn} \label{def: infty-delta type}
	Given   $0<\delta \leq 1$ we say that   $T$  is an \textbf{exponential Schr\"odinger--Calder\'on--Zygmund operator of $(\infty, \delta)$ type} if 
	\begin{enumerate}
		\item $T$ is bounded on $L^p(\mathbb{R}^d)$ for every $p>1$; 
		
		\item $T$ has an associated kernel $K$ verifying the following conditions:
		\begin{enumerate}
			\item there exist constants $c,C >0$ and $m>0$ such that 
			\begin{equation}\label{eq: size pointwise}
				|K(x,y)|\leq 
				\frac{C}{|x-y|^{d}}\, \exp\left(-c \left(1+\frac{|x-y|}{\rho(x)}\right)^{m}\right), \quad \text{for every } x\neq y;
			\end{equation}
			\item there exist a constant $C>0$ such that
			\begin{equation}\label{eq: smooth pointwise}
				|K(x,y)-K(x_0,y)|\leq  \frac{C}{|x-y|^{d}}\left(\frac{|x-x_0|}{|x-y|}\right)^\delta, 
			\end{equation}
			for every $|x-y|>2|x-x_0|$. 
		\end{enumerate}
	\end{enumerate}
\end{defn}

\subsection{Weights associated to a critical radius function} \label{sec: weights}

The class $A^{\rho}_{p}$ introduced in \cite{BHS11} properly contains the classical Muckenhoupt weights. For ${1<p<\infty}$, they are denoted by $A^{\rho}_{p} =\bigcup_{\theta\geq 0}A^{\rho,\theta}_p$,
where $w\in A^{\rho,\theta}_p$ means that
\begin{align*}
	\left(\int_B w\right)^{\frac{1}{p}}\left(\int_B w^{-\frac{1}{p-1}}\right)^{\frac{p-1}{p}} \leq C|B| \left(1+ \frac{r}{\rho(x)}\right)^\theta
\end{align*}
for every ball $B=B(x,r)$ with center $x \in \mathbb{R}^d$ and radius $r>0$. Similarly, when $p=1$, we denote by $A^\rho_1 =\bigcup_{\theta\geq 0}A^{\rho,\theta}_1$, where $A_1^{\rho,\theta}$ collects weights $w$ such that 
\begin{equation*}
	\int_B w \leq C|B| \left(1+\frac{r}{\rho(x)}\right)^\theta \inf_B w,
\end{equation*}
for every ball $B=B(x,r)$.

Let us also consider $A^{\rho,\loc}_{p}$ as defined in \cite{BCH19}. That is, a weight $w$ belongs to the class $A^{\rho,\loc}_{p}$ for $1<p<\infty$ if there exists a constant $C>0$ such that 
\begin{equation*}
	\left(\int_{B} w\right)^{\frac{1}{p}}\left(\int_{B} w^{-\frac{1}{p-1}}\right)^{\frac{p-1}{p}} \leq C|B|
\end{equation*}
for balls $B \in \mathcal{B}_\rho$, and for $p=1$ if \begin{equation*}
	\int_B w \leq C|B| \inf_B w,
\end{equation*}
holds for every $B \in \mathcal{B}_\rho$.

The weights involved in our main result were given by J. Bailey in \cite{Bailey21}. For $1<p<\infty$ and $c \geq 0$ we denote by $H^{\rho}_{p,c}:=\bigcup_{m\geq 0}H^{\rho,m}_{p,c},$ where $H^{\rho,m}_{p,c}$ is the class of weights $w$ for which there exists a constant $C$ such that
\begin{align*}
	\left(\int_B w\right)^{\frac{1}{p}}\left(\int_B w^{-\frac{1}{p-1}}\right)^{\frac{p-1}{p}} \leq C|B| \exp \left(c\left(1+ \frac{r}{\rho(x)}\right)^m\right),
\end{align*}
for each ball $B=B(x,r)$.

When $p=1$, for $c \geq 0$ we denote by $H^{\rho}_{1,c}:= \bigcup_{m\geq 0}H^{\rho,m}_{1,c},$ where ${w\in H^{\rho,m}_{1,c}}$ means that
\begin{align*}
	\int_B w \leq C|B| \exp \left(c\left(1+ \frac{r}{\rho(x)}\right)^m\right)\inf_{B} w,
\end{align*}
for each ball $B=B(x,r)$.

Clearly, for any $1\leq p<\infty$, $H^{\rho,m_1}_{p,c_1}\subseteq H^{\rho,m_2}_{p,c_2}$ whenever $c_1\leq c_2$ and $m_1\leq m_2$.

Notice that when $c=0$, we recover the Muckenhoupt $A_p$ classes. Moreover, it is easy to see that the classes $H^{\rho}_{p,c}$ share several key properties with the $A_p$ weights, as we establish in Lemma~\ref{lem: H properties}.

It is also straightforward to verify that the weights in $H^{\rho,m}_{p,c}$ meet the following doubling condition.

\begin{defn}
	Let $\kappa \geq1$ and $c>0$. We denote by  $D^{\rho}_{\kappa,c}:= \bigcup_{m \geq 0}D_{\kappa,c}^{\rho,m}$, where $D_{\kappa,c}^{\rho,m}$ is the class of weights $w$ such that there exists a constant $C$ for which
	\begin{align*}
		w(B(x,R))\leq C \left(  \frac{R}{r}\right)^{d\kappa}\exp{\left(c\left(1+ \frac{R}{\rho(x)}\right)^m \right)}w(B(x,r))
	\end{align*}
	for all  $x \in \mathbb{R}^d$ and $r<R$.
	
\end{defn}

\begin{rem}\label{obs: H doubling}	If $w \in H^{\rho,m}_{p,c}$, by H\"older's inequality it easy to check that ${w \in D^{\rho,m}_{p,cp}}$.  
\end{rem}

We will also deal with the reverse H\"older classes of weights. In \cite[Lemma~5]{BHS11} the authors prove that a weight $w\in A^\rho_p$, $1\leq p<\infty$, verifies the following reverse H\"older inequality
\begin{align}\label{eq: RH gamma}
	\left(\frac{1}{|B|}\int_B w ^{\eta}\right)^{\frac{1}{\eta}}\leq C   \left(\frac{1}{|B|}\int_B w\right)\left(1+ \frac{r}{\rho(x)}\right)^\beta
\end{align}
for every ball $B=B(x,r)$ and some constants $\beta, C>0$ and $\eta>1$.

Reverse H\"older classes adapted to the context of weights with exponential growth will also be needed.

\begin{defn}\label{def: RH classes}
	Let $\eta >1$ and $c>0$. We denote by  $RH_{\eta,c}^\rho:=\bigcup_{m \geq 0} RH_{\eta,c}^{\rho,m}$, where $RH_{\eta,c}^{\rho,m}$ is defined as those weights $w$ such that 
	\begin{align}\label{eq: RH exp}
		\left(\frac{1}{|B|}\int_B w ^{\eta}\right)^{\frac{1}{\eta}}\leq C    \left(\frac{1}{|B|}\int_B w\right)\exp \left(c\left(1+ \frac{r}{\rho(x)}\right)^m\right)
	\end{align}
	for every ball $B=B(x,r)$ and some constant $C$ independent of $B$.
\end{defn}
Clearly, \eqref{eq: RH gamma} implies \eqref{eq: RH exp}. 

\begin{rem}
	Note that when taking $m=0$, $RH^{\rho,0}_{p,c}$ and $D_{\kappa,c}^{\rho,0}$ are the classical  reverse H\"older classes and doubling classes, respectively.
\end{rem}

\section{Auxiliary results}

\subsection{Some properties for $H _{p,c}^{\rho}$}\label{subsec: weights properties}

The importance of the following result lies in the fact that the class of weights involved in  Theorem~\ref{thm: extrapolation H weights} generalizes those presented in \cite{BHQ19}. Additionally, the second inclusion enables the application of established tools for $A^{\rho,\loc}_p$. While this inclusion was originally obtained in \cite{Bailey21}, we provide a direct proof that avoids the need for intermediate weight classes and eliminates certain parameter restrictions. 

\begin{prop}\label{prop: classes inclusions} 
	Let ${1\leq p}<\infty$. For any $c>0$ we have
	\begin{equation*}
		A^\rho_p \subsetneq	H_{p,c}^{\rho} \subseteq A^{\rho,\loc}_{p}.
	\end{equation*}
\end{prop}

\begin{proof}
	The first inclusion can be found in \cite[Proposition~3.2]{Bailey21}.
	
	For the second inclusion, we consider $w \in H _{p,c}^{\rho,m}$ and a ball $B=B(x,r)$ in $\mathcal{B}_\rho$. {If $p>1$,} we get
	\begin{align*}
		\left(\int_B w\right)^{\frac{1}{p}}\left(\int_B w^{-\frac{1}{p-1}}\right)^{\frac{p-1}{p}}&\leq C|B|\exp\left(c\left(1+\frac{r}{\rho(x)}\right)^m\right)\\
		&\leq C|B|\exp\left(c2^m\right)\leq C|B|,
	\end{align*}
	where the constant is independent of $B$. Hence, $w\in A^{\rho,\loc}_{p}$. 
	
	{If $p=1$, we can obtain the inclusion in a similar way.}
	
	For the critical radius function $\rho(x)=\min\left\{1,\frac{1}{|x|}\right\}$ (which arises when dealing with the harmonic oscillator whose potential is $V(x)=|x|^2$), we shall see that $\omega(x)=e^{|x|}$ belongs to $H_{p,c}^{\rho}$ for some $c>0$, but $w\notin A_p^\rho$.
	
	Let $B=B(y,r)$ be a ball for $y\in \mathbb{R}^d$ and $r>0$, {and consider $p>1$}. Then,
	\begin{equation}\label{eq: exp weight bound}
		\left(\int_{B} e^{|x|} dx\right)^{\frac1p}\leq |B|^{\frac1p} e^{\frac{|y|+r}{p}}
	\end{equation}
	and
	\[\left(\int_{B} e^{-\frac{|x|}{p-1}} dx\right)^{\frac{p-1}{p}}\leq |B|^{1-\frac{1}{p}} e^{-\frac{|y|-r}{p}},\]
	imply that
	\[\left(\int_{B} e^{|x|} dx\right)^{\frac1p}\left(\int_{B} e^{-\frac{|x|}{p-1}} dx\right)^{\frac{p-1}{p}}\leq |B|e^{\frac{2}{p}r}.\]
	Now, notice that $r\leq \frac{r}{\rho(y)}$ for every $y\in \mathbb{R}^d$ and $r>0$, since $r=\frac{r}{\rho(y)}$ when $|y|\leq 1$ and $r\leq r|y|= \frac{r}{\rho(y)}$ when $|y|> 1$. Thus,
	\[\left(\int_{B} e^{|x|} dx\right)^{\frac1p}\left(\int_{B} e^{-\frac{|x|}{p-1}} dx\right)^{\frac{p-1}{p}}\leq |B|\exp\left(\frac2p \left(1+\frac{r}{\rho(y)}\right)\right)\]
	meaning that $\omega\in H_{p, 2/p}^{\rho}$. {When $p=1$, one can see that \eqref{eq: exp weight bound} still holds. Besides,
		\[\inf_{B} e^{|x|}\geq e^{|y|-r},\]
		which gives, as before, 
		\[\int_{B} e^{|x|} dx\leq |B|e^{2r}\inf_B e^{|x|}\leq |B|\exp\left(2 \left(1+\frac{r}{\rho(y)}\right)\right)\inf_B e^{|x|},\]
		and $\omega\in H_{1,2}^{\rho}$.}
	
	In order to see that $\omega\notin A_p^{\rho}$, we will follow the ideas given in \cite[p.~367]{Bailey18}. {Fix $p>1$.} Let us consider the balls $B_\ell=B(0,\ell)$ for $\ell>1$. On one hand, we can obtain
	\[\left(\int_{B_{2\ell}} e^{|x|} dx\right)^{\frac1p}\geq \left(\int_{B_{2\ell}\setminus B_\ell} e^{|x|} dx\right)^{\frac1p}=\left(\int_\ell^{2\ell} e^t t^{d-1} dt\right)^{\frac1p} 
	\left(e^{2\ell}-e^\ell\right)^{\frac1p}.\]
	On the other hand,
	\begin{align*}
		\left(\int_{B_{2\ell}} e^{-\frac{|x|}{p-1}} dx\right)^{\frac{p-1}{p}}&\geq \left(\int_{B_{2\ell}\setminus B_\ell} e^{-\frac{|x|}{p-1}} dx\right)^{\frac{p-1}{p}}\\
		&\gtrsim \left(e^{-\frac{\ell}{p-1}}-e^{-\frac{2\ell}{p-1}}\right)^{\frac{p-1}{p}}\gtrsim e^{-\frac{\ell}{p}}.
	\end{align*}
	Combining these inequalities, we obtain 
	\[\left(\fint_{B_{2\ell}} e^{|x|} dx\right)^{\frac1p}\left(\fint_{B_{2\ell}} e^{-\frac{|x|}{p-1}} dx\right)^{\frac{p-1}{p}}\gtrsim {\ell^{-d}}(e^\ell-1)^{\frac1p},\]
	and since this estimate holds for any $\ell>1$, the left-hand-side cannot be bounded by any polynomial on $\ell$. That is, $\omega\notin A_p^{\rho,\theta}$ for any $\theta\geq 0$, so $\omega\notin A_p^{\rho}$ as we claimed.
	
	{In the case $p=1$, 
		\[\int_{B_{2\ell}} e^{|x|} dx\geq e^{2\ell}-e^\ell\geq e^\ell, \qquad \inf_{B_{2\ell}} e^{|x|}=1,\]
		so for every $\ell>1$,
		\[\fint_{B_{2\ell}} e^{|x|} dx\left(\inf_{B_{2\ell}} e^{|x|}\right)^{-1}\gtrsim \ell^{-d}e^\ell.\]
		As before, we cannot find any polynomial on $\ell$ that bounds above the left-hand side, which means $\omega\notin A_1^{\rho,\theta}$ for any $\theta\geq 0$.}
\end{proof}
As it was observed in \cite{BCH19}, the $A^\rho_p$ classes preserve the favorable properties of the $A_p$ weights. Similarly,  as the reader may verify, the $H^{\rho}_{p,c}$ classes also exhibit these properties.

\begin{lem}\label{lem: H properties} The following properties hold:
	\begin{enumerate}
		\item $H^{\rho,m}_{p,c}\subset H^{\rho,m}_{q,c}$ for every $1\leq p\leq q<\infty$ and $m,c\geq 0$.
		\item If $w\in H^{\rho,m}_{p,c}$ for some $1<p<\infty$ and $m\geq 0$, then $\sigma:=w^{1-p'}\in H^{\rho,m}_{p',c}$.
		\item If $w_1\in H^{\rho,m_1}_{1,c_1}$ and $w_2\in H^{\rho,m_2}_{1,c_2}$ for some $m_1,m_2,c_1,c_2\geq 0$, then for every $1\leq p<\infty$ there exist $m,c\geq 0$ such that $w_1 w_2^{1-p}\in H^{\rho,m}_{p,c}$.
	\end{enumerate}
\end{lem}

Next lemma states that $H^{\rho}_{p,c}$ weights satisfy the reverse H\"older condition given in Definition~\ref{def: RH classes} above. The proof follows the lines of  \cite[Lemma~5]{BHS11}, with some modifications that we will provide.

\begin{lem}\label{lem: H in RH}
	Given $1\leq p<\infty$ and $c>0$, for any $w \in H^{\rho}_{p,c}$, there exist constants  $\eta >1$ and $c^*>0$ such that $w\in RH^{\rho}_{\eta,c^*}.$ 
\end{lem}
\begin{proof}
	Since $w \in H^{\rho,m}_{p,c}$ for some $m \geq 0$, by Proposition \ref{prop: classes inclusions} we get $w \in A^{\rho,\loc}_{p}$. {Then, according to \cite[p.~574]{BHS11}, there exists $\eta>1$ such that
		\begin{equation*}
			\left(\frac{1}{|B|}\int_B w ^{\eta}\right)^{\frac{1}{\eta}}\lesssim \frac{1}{|B|}\int_B w
		\end{equation*}
		for every $B\in \mathcal{B}_\rho$. Thus, we have the estimate for any ball $B=B(x,r)$ with $r\leq \rho(x)$, and any $c^*,m^*\geq 0$.}

	{Assume now that $r>\rho(x)$.} We follow the proof given in \cite{BHS11}, where an important tool is their Proposition 2. It guarantees the existence of a sequence of points $\left\{x_j\right\}_{j \in \mathbb{N}}$$\subset \mathbb{R}^d$ that satisfies the following two properties:
	\begin{itemize}
		\item $\mathbb{R}^d= \bigcup_{j \in \mathbb{N}} B(x_j,\rho(x_j))$;
		\item there exist constants $C,N_1>0$ such that for every $\sigma \geq 1$, 
		\[\sum_{j \in \mathbb{N}}\chi_{B(x_j,\sigma\rho(x_j))}\leq C \sigma^{N_1}.\]
	\end{itemize} 
	When estimating the integral of $w^\eta$ on $B$, for $\eta>1$, the same arguments given in the proof of \cite[Lemma~5]{BHS11} can be applied to get that
	\begin{equation}\label{eq: average w eta}
		\left(\int_B w ^{\eta}\right)^{\frac{1}{\eta}}\lesssim w(c(x,r)B) \rho(x)^{-\frac{d }{\eta'}}\left(1+\frac{r}{\rho(x)}\right)^{\frac{dk_0}{\eta'}},
	\end{equation}
	where $c(x,r)={8}C_0^2\left(1+\frac{r}{\rho(x)}\right)^{k_0/(k_0+1)}$.
	
	In the context of our class of weights, we use Remark~ \ref{obs: H doubling} to obtain  	\begin{equation*}
		w(c(x,r)B) \lesssim \exp\left(\tilde{c}\left(1+ \frac{r}{\rho(x)}\right)^{\tilde{m}}\right) w(B),
	\end{equation*}
	where  {$\tilde{c}\leq cp(8C_0^4)^m+\frac{dp}{m}$} and $\tilde{m}=m\frac{k_0}{k_0+1}+m$. Therefore, {plugging this into \eqref{eq: average w eta}} we have \eqref{eq: RH exp} {with} {$c^*\leq \tilde{c}+\frac{d(k_0+1)}{\tilde{m}\eta'}$} and {$m^*=\tilde{m}$.}  
\end{proof}

We also have the following property for $RH^{\rho}_{\eta,c}$ families of weights, which gives us an openness property. It is important to clarify that in the hypotheses of the following lemma the parameters $c$ and $c_1$ are not necessarily the same, which broadens the result.
\begin{lem}
	Let $w \in H^{\rho}_{p,c}\cap RH^{\rho}_{\eta,c_1}$ for some $\eta> 1$, $c,c_1>0$ and $p>1$. There exist $\beta>1$  and $\widetilde{c}>0$ such that $w \in RH^{\rho}_{\eta\beta, \widetilde{c}}$.
\end{lem}

\begin{proof}
	Let $w \in H^{\rho}_{p,c} \cap RH_{\eta,c_1}^\rho  $. Then, there exist  $m ,m_1> 0$ such that $w \in H^{\rho,m}_{p,c} \cap RH_{\eta,c_1}^{\rho, m_1}$.	
	Using both conditions we get $w^\eta  \in  H^{\rho}_{q,c_2} $ with $c_2>0$ and  $q = \eta (p-1)+1$. Indeed,
	\begin{align*}
		&\left(  \frac{1}{|B|} \int_B w ^\eta\right)^{\frac{1}{\eta(p-1)+1}}\left(\frac{1}{|B|}\int_B w ^{-\frac{\eta}{\eta(p-1)}}\right)^{\frac{\eta(p-1)}{\eta(p-1)+1}}\\
		& \lesssim \left( \exp \left(c_1\left(1+ \frac{r}{\rho(x)}\right)^{m_1}\right)\frac{1}{|B|}\int_B w\right)^{\frac{\eta}{\eta(p-1)+1}} \left(\frac{1}{|B|}\int_B w ^{-\frac{1}{(p-1)}}\right)^{\frac{\eta(p-1)}{\eta(p-1)+1}}\\
		& \lesssim  \exp \left(c_1\frac{\eta}{\eta(p-1)+1}\left(1+ \frac{r}{\rho(x)}\right)^{m_1}\right)\\
		&\quad \times \left(\left(\frac{1}{|B|}\int_B w\right)^{\frac{1}{p}}\left(\frac{1}{|B|}\int_B w ^{-\frac{1}{(p-1)}}\right)^{\frac{p-1}{p}}\right)^{\frac{\eta p}{\eta(p-1)+1}}\\
		&\lesssim \exp \left(c_1\frac{\eta}{\eta(p-1)+1}\left(1+ \frac{r}{\rho(x)}\right)^{m_1}\right)\exp\left( c\frac{\eta p}{\eta(p-1)+1} \left(1 + \frac{r}{\rho(x)}\right)^{m}\right)\\
		& \lesssim \exp \left(c_2\left(1+ \frac{r}{\rho(x)}\right)^{m_2}\right),
	\end{align*}
	with $c_2=\frac{(c_1+cp)\eta}{\eta(p-1)+1}>0$ and $m_2=\max\{m,m_1\}$.
	
	Then, from Lemma~\ref{lem: H in RH} there exists $\beta >1$ and $c'>0$ such that   $w^\eta \in RH^{\rho}_{\beta,c'}$. Using again that $w\in RH^\rho_{\eta,c_1}$, we get $w \in RH^\rho_{{\eta\beta},\widetilde{c}}$ with $\widetilde{c}=c_1+\frac{c}{\eta}>0$. 
\end{proof}

\subsection{Maximal operators associated with a critical radius function}\label{subsec: maximal operators}

For $f \in L^1_{\loc}\left(\mathbb{R}^d\right)$, we recall the definition of the local maximal operator over sub-critical balls, studied in \cite{BHS11} and \cite{BCH19}:
\begin{equation}
	M^{\loc}_\rho f(x)=\sup_{x \in B\in \mathcal{B}_\rho}\frac{1}{|B|}\int_B|f(y)|dy.
\end{equation}
As it was proved in \cite[Theorem~1]{BHS11}, for any $1<p<\infty$, $w\in A^{\rho,\loc}_p$ if and only if $M^{\loc}_\rho$ is bounded on $L^p(w)$.

In a similar way, the classes of weights $H_{p,c}^\rho$ are related with the boundedness of certain maximal functions that capture the exponential growth. These operators were already considered by J. Bailey in \cite{Bailey21}, but we define them below for the sake of completeness.

\begin{defn}
	Let $c, m\geq 0$ and let $\rho$ be a critical radius function. For $ f \in L^1_{\loc}(\mathbb{R}^d)$ and $x \in \mathbb{R}^d$ we consider the maximal operator given by
	\begin{equation}
		\widetilde{\mathcal{M}}^{m}_{\rho,c} f(x):=\sup_{B(x',r)\ni x}\frac{1}{\exp{\left(c\left(1+\frac{r}{\rho(x')}\right)^m\right)}}\fint_{B(x',r)}|f(y)|dy,
	\end{equation}
	and its centered version
	\begin{equation}
		\widetilde{M}^{m}_{\rho,c}f(x):=\sup_{r>0}\frac{1}{\exp{\left(c\left(1+\frac{r}{\rho(x)}\right)^m\right)}}\fint_{B(x,r)}|f(y)|dy.
	\end{equation}
\end{defn}

For any $x \in \mathbb{R}^d$ and $f \in L^{1}_{\loc}(\mathbb{R}^d)$, the inequality $\widetilde{M}^{m}_{\rho,c}f \lesssim \widetilde{\mathcal{M}}^{m}_{\rho,c}f$ trivially holds. The following proposition states that the reverse inequality also holds under some restrictions on the parameters. 

\begin{prop}[{\cite[Proposition~3.5]{Bailey21}}]\label{prop: maximal comparison}
	Let $\rho$ be a critical radius function with constants $C_0$ and $k_0$ as in Definition~\ref{def: critical radius function}. Given $c_1,c_2,m_1,m_2>0$ with   $m_1\geq (k_0+1)m_2$ and $c_1\geq c_2(2C_0)^{m_2}$ we have
	\begin{equation}\label{eq: maximal comparison}
		\widetilde{\mathcal{M}}_{\rho,c_1}^{m_1} f(x)\lesssim \widetilde{M}_{\rho,c_2}^{m_2} f(x)
	\end{equation}
	for every $f \in L^1_{\loc}(\mathbb{R}^d)$ and $x \in \mathbb{R}^d$. 
\end{prop}

The theorem below is significant in its own right, since it allows us to characterize the weight classes {$H^{\rho,m}_{p,c}$} as those that ensure the boundedness of the maximal operators $\widetilde{M}^{m}_{\rho,c}$, {but not necessarily with the same parameters $c$ and $m$.} This result will be useful to give an extrapolation tool in \S\ref{sec: proofs main results}.

{\begin{thm}\label{eq: maximal bounded Lpw}
		Let $1<p<\infty$. 
		\begin{enumerate} 
			\item If $w \in H^{\rho,m_1}_{p,c_1}$ for some $c_1,m_1\geq 0$, then  $\widetilde{M}^{m_1}_{\rho,c_2}$ is bounded on $L^p(w)$ for every $c_2 > c_ 1(8C_0)^{m_1}$.
			\item If $\widetilde{M}^{m_2}_{\rho,c_2}$ is bounded on $L^p(w)$ for some $c_2,m_2\geq 0$, then $w \in H^{\rho,m_1}_{p,c_1}$ for any $m_1\geq (k_0+1)m_2$ and $c_1\geq c_2(2C_0)^{m_2}$.
		\end{enumerate}
\end{thm}}

\begin{proof} The necessity of the class {$H^{\rho,m_1}_{p,c_1}$} was already stated in \cite[Proposition~3.6]{Bailey21} using \eqref{eq: maximal comparison}. We shall prove the sufficiency.

	{We fix $c_1,m_1\geq 0$ such that $w \in H^{\rho,m_1}_{p,c_1}$, and take $c_2\geq 0$ to be chosen later.}
	
	We can split the maximal operator as follows, 
	\begin{equation}
		\widetilde{M}^{m_1}_{\rho,c_2}f(x)\leq \widetilde{M}^{m_1,(1)}_{\rho,c_2}f(x)+\widetilde{M}^{m_1,(2)}_{\rho,c_2}f(x),
	\end{equation}
	where
	\begin{equation*}
		\displaystyle \widetilde{M}^{m_1,(1)}_{\rho,c_2}f(x)=\sup_{r\leq \rho(x)}\frac{1}{\exp \left(c_2\left(1+\frac{r}{\rho(x)}\right)^{m_1}\right)}\fint_{B(x,r)}|f(y)|dy
	\end{equation*}  
	and 
	\begin{equation*}
		\widetilde{M}^{m_1,(2)}_{\rho,c_2}f(x)=\sup_{r>\rho(x)}\frac{1}{\exp \left(c_2\left(1+\frac{r}{\rho(x)}\right)^{m_1}\right)}\fint_{B(x,r)}|f(y)|dy.
	\end{equation*}
	
	By Proposition~\ref{prop: classes inclusions}, ${H^{\rho,m_1}_{p,c_1}} \subset A^{\rho,\loc}_p$, and since  $\widetilde{M}^{m_1,(1)}_{\rho,c_2}f(x)\leq  M^{\loc}_\rho f(x)$ {for any $c_2\geq 0$}, we deduce that $\widetilde{M}^{m_1,(1)}_{\rho,c_2}$ is bounded on $L^p(w)$ for any {$c_2\geq 0$ and $w \in H^{\rho,m_1}_{p,c_1}$.}

	We will now deal with $\widetilde{M}^{m_1,(2)}_{\rho,c_2}f$. Let $\left\{Q_k\right\}_{k\geq 1}$ be the covering given in \cite[Proposition~2]{BHS11} of critical balls $Q_k=B(x_k,\rho(x_k))$. We fix $x\in \mathbb{R}^d$ and set $R_j=\left\{ r:2^{j-1}\rho(x)<r\leq 2^j\rho(x)\right\}$, {$j\geq 1$}. Then, for any $x\in Q_k$,  {from Definition~\ref{def: critical radius function} is $\rho(x_k)\leq C_02^{k_0}\rho(x)$} so we have
	\begin{align*}
		\widetilde{M}^{m_1,(2)}_{\rho,c_2}f(x)&=\sup_{r>\rho(x)}\frac{1}{\exp \left(c_2\left(1+\frac{r}{\rho(x)}\right)^{m_1}\right)}\fint_{B(x,r)}|f(y)|dy\\
		& =\sup_{j\geq1}\sup_{r \in R_j}\frac{1}{\exp \left(c_2\left(1+\frac{r}{\rho(x)}\right)^{m_1}\right)\omega_d r^d}\int_{B(x,r)}|f(y)|dy\\
		& \lesssim \sup_{j\geq1}\frac{1}{\exp \left(c_2\left(1+2^{j-1}\right)^{m_1}\right)}\frac{1}{\left(2^{j-1}\rho(x)\right)^d}\int_{B(x,2^j\rho(x))}|f(y)|dy\\
		&\lesssim \sup_{j\geq1}\frac{2^{-jd}}{\exp \left(c_2\left(1+2^{j-1}\right)^{m_1}\right)\rho(x_k)^d}\int_{C_jQ_k}|f(y)|dy,
	\end{align*}
	where {$C_j=2^{j+2} C_0$} can be obtained from \eqref{eq: critical radius ineq}. 
	
	Finally, from H\"older's inequality and {the properties of the covering}, we get
	\begin{align*}
		\int_{\mathbb{R}^d}\left|\widetilde{M}^{m_1,(2)}_{\rho,c_2}f\right|^p w 
		&\leq \sum_{k\geq 1}\int_{Q_k} \left|\widetilde{M}^{m_1,(2)}_{\rho,c_2}f\right|^p w \\
		&\leq \sum_{k\geq 1} \sup_{j\geq1}\frac{2^{-jdp}}{\exp \left(c_2p\left(1+2^{j-1}\right)^{m_1}\right)\rho(x_k)^{dp}}\left(\int_{C_jQ_k}|f|\right)^p \left(\int_{Q_k}w\right)\\
		& \leq \sum_{k\geq 1} \sup_{j\geq1}\frac{2^{-jdp}}{\exp \left(c_2\left(1+2^{j-1}\right)^{m_1}\right)\rho(x_k)^{dp}}\left(\int_{C_jQ_k}|f|^p w \right)\left(\int_{C_jQ_k}w^{-\frac{1}{p-1}}\right)^{p/p'}\int_{C_jQ_k}w\\
		& \leq \sum_{k\geq 1} \sup_{j\geq1}\frac{2^{-jdp}\left(C_j \rho(x_k)\right)^{dp} \exp\left(c_1p\left(1+\frac{C_j\rho(x_k)}{\rho(x_k)}\right)^{m_1} \right)}{\exp \left(c_2p\left(1+2^{j-1}\right)^{m_1}\right)\rho(x_k)^{dp}}\int_{C_jQ_k}|f|^p w\\
		& \leq {\left(4C_0\right)^{dp}}\sum_{k\geq 1} \sup_{j\geq1}\frac{ \exp\left(c_1p\left(1+{2^{j+2}C_0}\right)^{m_1} \right)}{\exp \left(c_2p\left(1+2^{j-1}\right)^{m_1}\right)}\int_{C_jQ_k}|f|^p w\\
		&\lesssim \sum_{j\geq 1}{\exp\left(\tilde{c}\left(1+2^{j}\right)^{m_1}\right)}\sum_{k\geq 1}\int_{C_jQ_k}|f|^pw,
	\end{align*}
	{where $\tilde{c}= c_1 p(8C_0)^{m_1} - c_2p$.}
	
	Since $\sum_{k\geq 1} \chi_{C_jQ_k}\leq C\, C_j^{N_1}\sim C2^{jN_1}$, 
	\begin{equation*}
		\int_{\mathbb{R}^d}\left|\widetilde{M}^{m_1,(2)}_{\rho,c_2}f\right|^p w 
		\lesssim \int_{\mathbb{R}^d} |f|^p w \sum_{j\geq 1}{\exp\left(\tilde{c}\left(1+2^{j}\right)^{m_1}\right)}2^{jN_1}\lesssim \int_{\mathbb{R}^d} |f|^p w,
	\end{equation*}
	provided that {$c_2 > c_ 1(8C_0)^{m_1}$} is chosen. 
\end{proof}

{\begin{rem}\label{obs: maximal p=1}
		For $p=1$, we have that if $w$ is a weight verifying $\widetilde{\mathcal{M}}_{\rho,c}^{m}w\lesssim w$ a.e., then $w\in H_{1,c}^{\rho,m}$. Note that the parameters $c$ and $m$ are preserved. This can be deduced by adapting the classical argument for the Hardy-Littlewood maximal operator and $A_1$ weights. 
\end{rem}}

\section{Boundedness results for exponential Schr\"odinger--Calder\'on--Zygmund operators} \label{sec: proofs main results}

We first present useful tools for proving our main theorem. While many of them follow the approach outlined in \cite{BCH19}, additional considerations may be necessary in this new context. For this reason, we will provide details whenever clarification is required.

{\begin{thm}\label{thm: RdF extrapolation H weights}
		Let $q>0$ and let $(f,g)$ be a pair of measurable and nonnegative functions such that
		\begin{equation}\label{eq: weighted Linfty ineq}
			\|fw\|_\infty \leq C\|gw\|_\infty
		\end{equation}
		holds for every weight $w$ such that $w^{-q}\in H_{1,c_1}^{\rho,m_1}$ for certain $c_1,m_1\geq 0$, whenever the left hand side is finite, with a constant $C$ depending on $w$ only through the constant of the class $H_{1,c_1}^{\rho,m_1}$. 
		
		Then, for any $q<p<\infty$,
		\[\|f\|_{L^p(w)}\leq C\|g\|_{L^p(w)}\]
		for every $w\in H_{p/q,c}^{\rho, m}$ with $m=\frac{m_1}{k_0+1}$ and $c<c_1 (4C_0)^{-2m}$, whenever the left hand side is finite, with a constant $C$ depending on $w$ only through the constant of the class $H_{p/q,c}^{\rho,m}$.
\end{thm}}

{\begin{proof}
		Fix $1<p<\infty$ and assume first $q=1$. Without loss of generality, we can suppose $f,g\in L^p(w)$. Then, from \cite[Lemma~1]{BCH13Extrapol} (see also \cite{HMS}), there exist positive functions $F,G\in L^p(w^{-1/(p-1)})$ such that
		\[\|F\|_{L^p(w^{-1/(p-1)})}\leq 2, \quad \|G\|_{L^p(w^{-1/(p-1)})}\leq 2\]
		and
		\[\|f\|_{L^p(w)}=\|fw^{1/(p-1)}F^{-1}\|_{\infty},\quad \|g\|_{L^p(w)}=\|gw^{1/(p-1)}G^{-1}\|_{\infty}.\]
		
		Let $w\in H^{\rho,m}_{p,c}$ with $m=\frac{m_1}{k_0+1}$ and $c<c_1 (4C_0)^{-2m}$. We will proceed similarly to the proof of \cite[Theorem~3]{BCH13Extrapol}, but we will give the essential changes. Let $h=F+G$ and define
		\[\mathscr{R}h(x)=\sum_{k=0}^{\infty} \frac{\widetilde{T}_{c_1,m_1}^k h(x)}{2^k\|\widetilde{T}_{c_1,m_1}\|_{L^p(w^{-1/(p-1)})}^k},\]
		where 
		\[\widetilde{T}_{c_1,m_1}f=\widetilde{\mathcal{M}}_{\rho,c_1}^{m_1}(fw^{-1/(p-1)})w^{1/(p-1)},\]
		which is a bounded operator on $L^p(w^{-1/(p-1)})$ since $\widetilde{\mathcal{M}}_{\rho,c_1}^{m_1}$ is bounded on $L^p(w)$. Indeed, as $m_1=m(k_0+1)$ and taking $c_2=c_1(2C_0)^{-m}$, $\widetilde{\mathcal{M}}_{\rho,c_1}^{m_1}\lesssim \widetilde{M}_{\rho,c_2}^{m}$ by Proposition~\ref{prop: maximal comparison}. And, noting that $c_2>c(8C_0)^m$, Theorem~\ref{eq: maximal bounded Lpw} gives the desired $L^p(w)$-boundedness of $\widetilde{\mathcal{M}}_{\rho,c_1}^{m_1}$.
		
		The function $\mathscr{R}h$ verifies 
		\[h\leq \mathscr{R}h, \quad \|\mathscr{R}h\|_{L^p(w^{-1/(p-1)})}\leq 2\|h\|_{L^p(w^{-1/(p-1)})}\]
		and 
		\[\widetilde{T}_{c_1,m_1}(\mathscr{R}h)\leq 2\|\widetilde{T}_{c_1,m_1}\|_{L^p(w^{-1/(p-1)})}\mathscr{R}h,\]
		where this last inequality implies that $(\mathscr{R}h) w^{-1/(p-1)}$ is a weight in $H_{1,c_1}^{\rho,m_1}$ by Remark~\ref{obs: maximal p=1} and the constant $2\|\widetilde{T}_{c_1,m_1}\|_{L^p(w^{-1/(p-1)})}$ is not smaller than the constant of $(\mathscr{R}h) w^{-1/(p-1)}$ on the corresponding class of weights.
		
		The proof can be concluded as the proof of \cite[Theorem~3]{BCH13Extrapol}.
		
		If $q>0$, we take $w^{-1}\in H_{1,c_1}^{\rho,m_1}$, so we can use the hypothesis with $w^{1/q}$ since $(w^{1/q})^{-q}\in H_{1,c_1}^{\rho,m_1}$. As in the proof of \cite[Corollary~4]{BCH13Extrapol}, we apply the result obtained for the case $q=1$ to the pairs $(f^q,g^q)$ to get
		\[\|f\|_{L^p(w)}\leq C\|g\|_{L^p(w)}\]
		for any $q<p<\infty$ and $w\in H_{p/q,c}^{\rho,m}$ with $c$ and $m$ as before.
\end{proof}}

In order to get the boundedness on $L^p(w)$ for exponential Schr\"odinger--Calder\'on--Zygmund operators, we state an extrapolation result from the extreme point $p_0=\infty$. Since many of the operators arising on the Schr\"odinger context are not bounded on $L^\infty$, the following result will allow us to include them by establishing their continuity from $L^\infty(w)$ to $\BMO(w)$, for certain classes of weights. 

{\begin{thm}\label{thm: extrapolation H weights}
		Let $q>0$ and let $T$ be a bounded operator from $L^\infty(w)$ into $BMO_\rho(w)$ for any weight $w$ such that $w^{-q}\in H^{\rho,m_1}_{1,c_1}$ for certain $c_1,m_1\geq 0$, with constant depending on $w$ through the constant of the condition   $w^{-q}\in H_{1,c_1}^{\rho, m_1}$. Then, $T$ is bounded on $L^p(w)$ for every $q<p<\infty$ and every $w\in H_{p/q,c}^{\rho,m}$ where $m=\frac{m_1}{k_0+1}$ and $c<c_1(4C_0)^{-2m}$. 
	\end{thm}
	
	\begin{proof}
		We  first observe that from Lemma~\ref{lem: BMO Msharp} and the hypothesis on $T$, for every $f\in L^\infty(w)$ we have 
		\begin{equation*}
			\| M_{\loc}^\sharp(Tf) w\|_\infty \leq \|Tf\|_{\BMO(w)}\leq C \|f\|_{L^\infty(w)}=C\|fw\|_\infty
		\end{equation*}
		whenever $w$ is a weight such that $w^{-q}\in H_{1,c_1}^{\rho,m_1}$. This means that the pair $(M^\sharp_\loc(Tf),f)$ verifies \eqref{eq: weighted Linfty ineq} for every $w^{-q}\in H_{1,c_1}^{\rho,m_1}$. Therefore, by Theorem~\ref{thm: RdF extrapolation H weights}
		\[\|M_{\loc}^\sharp(Tf) w\|_{L^p(w)}\leq C \|f w\|_{L^p(w)}\]
		for every $w\in H_{p/q,c}^{\rho,m}$ with $m=\frac{m_1}{k_0+1}$ and $c<c_1(4C_0)^{-2m}$, where the constant $C$ depends only on the constant of the weight on its class. 
		
		Finally, Proposition \ref{prop: classes inclusions} allows us to apply \cite[Corollary 5]{BCH13Extrapol} to finish the proof. 
\end{proof}}

The next two propositions show that, for certain parameters, the exponential Schr\"odinger--Calder\'on--Zygmund operators of $(s,\delta)$ type for   $1<s\leq \infty$ and $0<\delta\leq 1$, fall under the hypothesis of the theorem above.

\begin{prop}\label{prop: endpoint infty-delta}
	Let $T$ be an exponential Schr\"odinger--Calder\'on--Zygmund operator of $(\infty,\delta)$ type as given in Definition \ref{def: infty-delta type} with constants $c_1$ and $m_1$ in \eqref{eq: size pointwise}. Then, $T$ is bounded from $L^\infty(w)$ into $BMO_\rho(w)$ for every weight $w$ such that $w^{-1}\in H_{1,c}^{\rho,m_1}$ with $c<c_1{2^{-m_1}}$. 
\end{prop}

\begin{proof} 
	Let $x_0\in \mathbb{R}^d$, $r\leq \rho(x_0)$ and $B=B(x_0,r)$. We consider 
	\begin{equation}\label{eq: f1f2f3}
		f=f\chi_{2B} + f\chi_{B(x_0,2\rho(x_0))\setminus 2B} + f\chi_{B(x_0,2\rho(x_0))^c}=: f_1+f_2+f_3.
	\end{equation}
	Here, we need to observe that as  $w^{-1}\in H_{1,c}^{\rho,m_1}$, from Lemma \ref{lem: H in RH} there exist $\eta>1$ and 
	$c^*>0$ such that $w\in RH^{\rho}_{\eta,c^*}.$
	Therefore there exist a constant $C$ depending on $m_1$ and $c^*$ such that  
	\[\left(\frac{1}{|B|}\int_B w^{-\gamma}\right)^{1/\gamma}\leq C\frac{1}{|B|}\int_B w^{-1}\]
	since  $r\leq \rho(x_0)$. 
	
	Using that $T$ is bounded on $L^\gamma(\mathbb{R}^d)$, 
	and that $w^{-1}\in A_1^{\rho,\loc}$ (by Proposition~\ref{prop: classes inclusions}), we have 
	\begin{align*}
		\frac{1}{|B|} \int_B |T f_1(x)| dx
		&\leq 
		\left( \frac{1}{|B|} \int_B |T f_1(x)|^\gamma dx\right)^{1/\gamma}
		\\
		&\leq  C \left( \frac{1}{|B|} \int_{2B} | f(x)|^\gamma dx\right)^{1/\gamma}
		\\
		&\leq  C \| f w\|_\infty \left( \frac{1}{|2B|} \int_{2B}  w^{-\gamma}(x) dx\right)^{1/\gamma}
		\\
		&\leq  C \| f w\|_\infty \left( \frac{1}{|B|} \int_{B}  w^{-1}(x) dx\right)
		\\
		& \leq C \| f w\|_\infty \inf_{B} w^{-1}\\
		& =
		C \frac{\| f w\|_\infty}{\|  w\chi_B\|_\infty}.
	\end{align*}
	
	To estimate $|Tf_3(x)|$ for $x\in B$, we denote by $B_\rho=B(x_0,\rho(x_0))$ and  $B_\rho^k=2^k B_\rho$ for each $k\in \mathbb{N}$, {$k\geq 2$.} Thus	 
	\begin{align*}
		\int_{B_\rho^{k+1}}    |f| &\leq \| fw\|_{\infty} \int_{B_\rho^{k+1}} w^{-1} \leq C
		\exp\left(c \left(1+ 2^{k+1}\right)^{m_1}\right) |B_\rho^{k+1}|\frac{\| fw\|_{\infty}}{\|w\chi_{B_\rho^{k+1}}\|_{\infty}}
	\end{align*}
	for each $k\in \mathbb{N}$. Then, by the size condition \eqref{eq: size pointwise} on the kernel $K$ with $c_1$ and $m_1$, applying the above inequality and {Definition~\ref{def: critical radius function}}, we get
	\begin{align*}
		|Tf_3(x)| &\leq \sum_{{k\geq 2}} \int_{B_\rho^{k+1}\setminus B_\rho^k}|K(x,z)|| f(z)| dz
		\\
		& \leq C \sum_{{k\geq 2}} \int_{B_\rho^{k+1}\setminus B_\rho^k}   \frac{1}{|B_\rho^{k-1}|}\, \exp\left(-c_1 \left(1+ \frac{2^{k-1}\rho(x_0)}{\rho(x)}\right)^{m_1}\right) |f(z)| dz 
		\\
		& \leq C \sum_{{k\geq 2}} \exp\left(-c_1 \left(1+ 2^{{k-1}}C_02^{k_0}\right)^{m_1}\right)\frac{1}{|B_\rho^{k+1}|} \int_{B_\rho^{k+1}}    |f(z)| dz 
		\\
		& \leq C \sum_{{k\geq 2}}  \exp\left(-\left({\frac{c_1}{2^{m_1}}}-c\right) \left(1+ 2^{k+1}\right)^{m_1}\right)\frac{\| fw\|_{\infty}}{\|w\chi_{B_\rho^{k+1}}\|_\infty}
		\\
		& \leq C \frac{\| fw\|_{\infty}}{\|w\chi_{B}\|_\infty}\sum_{{k\geq 2}}  \exp\left(-\left({\frac{c_1}{2^{m_1}}}-c\right) \left(1+ 2^{k+1}\right)^{m_1}\right)
	\end{align*}
	where the series converges since $c<c_1{2^{-m_1}}$.
	
	The bound for $f_2$ follows as in the proof of \cite[Proposition~5]{BCH13Extrapol}, since only the smoothness condition is required, and the operator $T$ under consideration satisfies \cite[(41)]{BCH13Extrapol}, which is precisely \eqref{eq: smooth pointwise}.
\end{proof}

\begin{prop}\label{prop: endpoint s-delta}
	Let $T$ be an exponential Schr\"odinger--Calder\'on--Zygmund operator of $(s,\delta)$ type as given in Definition \ref{def: s-delta type} with constants $c_1$ and $m_1$ in \eqref{eq: size Hormander}. Then, $T$ is bounded from $L^\infty(w)$ into $\BMO(w)$ for every weight $w$ such that $w^{-s'}\in H_{1,c}^{\rho, m_1}$ with $c<s'c_1{2^{-m_1}}$. 
\end{prop}

\begin{proof}
	Let $x_0\in \mathbb{R}^d$, $r\leq \rho(x_0)$ and $B=B(x_0,r)$. We consider $f_1,f_2$ and $f_3$ as in \eqref{eq: f1f2f3}.
	
	Here, we need to observe that as  $w^{-s'}\in H_{1,c}^\rho$ and $r\leq \rho(x_0)$, the estimates of the average $\frac{1}{|B|} \int_B |Tf_1| dx$ follows using Kolmogorov's inequality and the hypothesis on $w$ as in {the proof of \cite[Proposition~6]{BCH13Extrapol}}. Then we have  
	\begin{align*}
		\frac{1}{|B|} \int_B |T f_1(x)| dx
		\leq 
		C \frac{\| f w\|_\infty}{\|  w\chi_B\|_\infty}.
	\end{align*}
	
	To estimate $|Tf_3(x)|$ for $x\in B$, we denote, as in the previous proof, $B_\rho=B(x_0,\rho(x_0))$ and $B_\rho^k=2^k B_\rho$ for $k\in \mathbb{N}${, $k\geq 2$}. From the condition on the weight we can obtain	
	\begin{align*}
		\left(\int_{B_\rho^{k+1}}    |f|^{s'}\right)^{1/s'} &\leq \| fw\|_{\infty} \left(\int_{B_\rho^{k+1}} w^{-s'} \right)^{1/s'}\leq 
		C\exp\left(\frac{c}{s'} \left(1+ 2^{k+1}\right)^{m_1}\right) \left|B_\rho^{k+1}\right|^{1/s'}\frac{\| fw\|_{\infty}}{\|w\chi_{B_\rho^{k+1}}\|_{\infty}}.
	\end{align*}
	From this inequality and \eqref{eq: size Hormander} with $c_1,m_1>0$, we estimate	
	\begin{align*}
		|Tf_3(x)| &\leq \sum_{{k\geq 2}} \int_{B_\rho^{k+1}\setminus B_\rho^k}|K(x,z)|| f(z)| dz
		\\
		& \leq \sum_{{k\geq 2}} \left( \int_{B_\rho^{k+1}\setminus B_\rho^k}|K(x,z)|^s dz \right)^{1/s} \left(  \int_{B_\rho^{k+1}} |f(z)|^{s'} dz\right)^{1/s'}
		\\
		& \leq C\sum_{{k\geq 2}} \rho(x_0)^{-d/s'} \exp\left(-c_1 \left(1+2^{{k}}\right)^{m_1}\right) 	\exp\left(\frac{c}{s'} \left(1+ 2^{k+1}\right)^{m_1}\right)\left|B_\rho^{k+1}\right|^{1/s'}\frac{\| fw\|_{\infty}}{\|w\chi_{B_\rho^{k+1}}\|_\infty}
		\\
		& \leq C\frac{\| fw\|_{\infty}}{\|w\chi_{B_\rho}\|_\infty} \sum_{{k\geq 2}} \exp\left(-\left({\frac{c_1}{2^{m_1}}}-\frac{c}{s'}\right) \left(1+ 2^{k+1}\right)^{m_1}\right)
	\end{align*}
	where the series converges as $c<s'c_1{2^{-m_1}}$.
	
	The bound for $f_2$ can be obtained as in the proof of \cite[Proposition~6]{BCH13Extrapol} since \eqref{eq: smooth Hormander} implies \cite[(45)]{BCH13Extrapol}.
\end{proof}

\section{Applications for operators associated to \texorpdfstring{$-\Delta +\mu$}{-Δ+μ}}

We consider a Schr\"odinger operator with measure $\mu$ in $\mathbb{R}^d$, with $d\geq 3$, 
\[\mathcal{L}_\mu=-\Delta +\mu, \]
where $\mu$ is a nonnegative Radon measure on $\mathbb{R} ^d$ that satisfies the following conditions: there exist constants   $\delta_\mu, C_\mu>0$ and $D_\mu\geq 1$ such that
\begin{align}\label{eq: prop 1 mu}
	\mu(B(x,r)) \leq C_\mu \left( \frac{r}{R}\right)^{d-2+\delta_\mu} \mu(B(x,R))
\end{align}
and 
\begin{align}\label{eq: prop 2 mu}
	\mu(B(x,2r))\leq D_\mu \left(\mu(B(x,r))+r ^{d-2}\right)
\end{align}
for all $x \in \mathbb{R}^d$ and $0<r<R$.

From  \eqref{eq: prop 1 mu} and  \eqref{eq: prop 2 mu}  it can be proved that 
\begin{equation}\label{eq: prop 1 integral mu}
	\int_{B(x,R)}\frac{d\mu(y)}{| y-x|^{d-2} }\leq C \frac{\mu(B(x,R))}{R^{d-2}}, 
\end{equation} 
and if $\delta_\mu>1$ then we also have 
\begin{equation}\label{eq: prop 2 integral mu}
	\int_{B(x,R)}\frac{d\mu(y)}{| y-x| ^{d-1}}\leq C\frac{\mu(B(x,R))}{R^{d-1}}, 
\end{equation}
for all $x \in \mathbb{R}^d$ and $R>0$.

For any nonnegative Radon measure $\mu$ on $\mathbb{R} ^d$ verifying conditions \eqref{eq: prop 1 mu} and \eqref{eq: prop 2 mu}, the function given by
\begin{equation}\label{eq: rho_mu}
	\rho_\mu(x):=\sup \left\{ r>0: \frac{\mu(B(x,r))}{r^{d-2}}\leq 1\right\}, \quad x \in \mathbb{R}^d,
\end{equation}
is a critical radius function (see Definition~\ref{def: critical radius function}). 

It is immediate from the definition of $\rho_\mu$ that for every $x \in \mathbb{R}^d$
\begin{equation} \label{eq: mu equiv critical balls}
	\frac{\mu(B(x,\rho_\mu(x)))}{\rho_\mu(x)^{d-2}}\sim 1.
\end{equation}
Now, we shall apply Propositions~\ref{prop: endpoint infty-delta} and~\ref{prop: endpoint s-delta} to some operators related to~$\mathcal{L}_\mu$.

It is worth noticing that, as a particular case of measure $\mu$ we will be able to consider $ {d\mu(x)=V(x) dx}$, with potential $V\geq 0$ in the reverse H\"older class $RH_{\frac{d}{2}}$, that is,
\begin{equation*}
	\left(\frac{1}{|B(x,r)|} \int_{B(x,r)} V(y)^{\frac d2 }dy\right)^{\frac 2d} \leq C \frac{1}{|B(x,r)|} \int_{B(x,r)} V(y)dy,
\end{equation*}
for every $x\in \mathbb{R}^d$ and $r>0$. Under this condition it is easy to see that 	$\mu$ satisfy \eqref{eq: prop 1 integral mu} and \eqref{eq: prop 2 integral mu} for some $\delta_\mu,C_\mu>0$ and $D_\mu\geq 1$.	In this context, we recover the operators studied in \cite{Shen95}.

Hereafter, we denote by  $\Gamma_\mu$ the fundamental solution of  $\mathcal{L}_\mu$ and by $\Gamma_{\mu+\lambda}$ the fundamental solution of $\mathcal{L}_{\mu+\lambda}=-\Delta + (\mu+\lambda)$ for $\lambda\geq 0$. As observed in \cite[p.~554]{Shen99}, the measure $\mu+\lambda$ verifies \eqref{eq: prop 1 mu} and \eqref{eq: prop 2 mu} for any $\lambda\geq 0$. Also, we may write $d_{\mu}$ and $d_{\mu+\lambda}$ to indicate the Agmon distances associated with the critical radius functions $\rho_\mu$ and $\rho_{\mu+\lambda}$ defined as in \eqref{eq: rho_mu} for the measures $\mu$ and $\mu+\lambda$, respectively. For the sake of simplicity, we shall henceforth write $\rho$ instead of $\rho_\mu$, with the understanding that the reference to $\mu$ will be clear from the context.

We refer to \cite[Theorems~0.8~and~0.17]{Shen99} for the following results.	Actually,  although the estimate \eqref{eq: grad fundamental solution bound} given below is not explicitly stated in \cite{Shen99}, it
is essentially contained within the proof of \cite[Theorem~0.19]{Shen99} (see also \cite[Theorem~4.2]{Bailey21}).

For a measure $\mu$ satisfying  \eqref{eq: prop 1 mu} and \eqref{eq: prop 2 mu}, there exist positive constants $C_1, C_2, C_ 3, \epsilon_1, \epsilon_2$ and $\epsilon_3$ such that 
\begin{equation}\label{eq: fundamental solution bounds}
	C_1 \frac{e^{-\epsilon_1 d_{\mu}(x,y)}}{|x-y|^{d-2}}\leq
	\Gamma_{\mu}(x,y)\leq C_2 \frac{e^{-\epsilon_2 d_{\mu}(x,y)}}{|x-y|^{d-2}}, 
\end{equation}
and 	
\begin{equation}\label{eq: grad fundamental solution bound}
	|\nabla_1\Gamma_{\mu}(x,y)|\leq C_3
	\frac{e^{-\epsilon_3 d_\mu(x,y)}}{|x-y|^{d-2}}\left(\int_{B \left(x,\frac{|y-x|}{2}\right)}\frac{d\mu(z)}{|z-x|^{d-1}}+ \frac{1}{|x-y|}\right)
\end{equation}
for $x\neq y $. Also, for $\delta_\mu>1$ we have that there exist $C,\epsilon>0$ such that 
\begin{equation*}
	| \nabla_1\Gamma_{\mu}(x,y)|\leq C \frac{e^{-\epsilon d_{\mu}(x,y)}}{|x-y|^{d-1}}, \quad \mbox{ for }  x\neq y. 
\end{equation*}

The following technical lemmas  will be  important tools for the examples below. 

\begin{lem}\label{lem: drho lower bound}
	Let $\rho:\mathbb{R} \rightarrow [0,\infty)$ be a critical radius function as in Definition~\ref{def: critical radius function}. Then, there exists a constant $C>0$, depending on $D_0, D_1$ and $k_0$, such that 
	\begin{equation}\label{eq: drho lower bound}
		d_\rho(x,y)\geq D_{1}^{-1}\left(1+\frac{|x-y|}{\rho(x)}\right)^{\frac{1}{k_0+1}}+D_0^{-1}\left(1+\frac{\rho(x)}{|x-y|}\right)^{-1}-C.
	\end{equation}
\end{lem}
\begin{proof} We will consider two cases. If $|x-y|\leq 2\rho(x)$, by Lemma~\ref{eq: drho equiv usual dist} we have
	\[d_\rho(x,y)\geq D_0^{-1}\frac{|x-y|}{\rho(x)}.\]
	Since 
	\[\frac{|x-y|}{\rho(x)}\geq \frac{|x-y|}{\rho(x)+|x-y|}=\left(1+\frac{\rho(x)}{|x-y|}\right)^{-1},\]
	we get
	\begin{equation*}
		d_\rho(x,y)\geq D_0^{-1}\left(1+\frac{\rho(x)}{|x-y|}\right)^{-1}.
	\end{equation*}
	As $1+\frac{|x-y|}{\rho(x)}\leq 3$, we can take $C\geq D_1^{-1}3^{\frac{1}{k_0+1}}$ to get the desired inequality in this case.

	On the other hand,  if $|x-y|>2\rho(x)$ we can apply Lemma~\ref{lem: drho comparable local} to get
	\[d_\rho(x,y)\geq D_1^{-1}\left(1+\frac{|x-y|}{\rho(x)}\right)^{\frac{1}{k_0+1}}.\]
	Trivially, $1+ \frac{\rho(x)}{|x-y|}\geq 1,$
	so choosing $C\geq D_0^{-1}$ gives \eqref{eq: drho lower bound} for $|x-y|>2\rho(x)$.
	Finally, taking $C=\max\left\{D_1^{-1}3^{\frac{1}{k_0+1}}, D_0^{-1}\right\}$, the proof is finished.
\end{proof}

The following is a generalization of a known property for the case where $d\mu(x) = V(x)dx$, as given in \cite[Lemma~1]{GLP08}. 

\begin{lem}	\label{lem: mu extra decay}	
	Let  $\mu$ be a measure verifying \eqref{eq: prop 1 mu} and \eqref{eq: prop 2 mu}. Then, there exist a constant $C>0$ such that for any $x_0\in \mathbb{R}^d$ and $R>0$ 
	\begin{equation*}
		\mu\left(B(x_0,R)\right) \leq C  R^{d-2} \left(1+\frac{R}{\rho(x_0)}\right)^{N},
	\end{equation*}
	for all $N\geq\log_2 D_\mu$, with $D_\mu$ as in \eqref{eq: prop 2 mu}. 
\end{lem}

\begin{proof}
	Let us first consider the case $R\leq \rho(x_0)$. By \eqref{eq: prop 1 mu} and \eqref{eq: mu equiv critical balls},
	\begin{align*}
		\mu\left(B(x_0,R)\right)\lesssim \left(\frac{R}{\rho(x_0)}\right)^{d-2+\delta_\mu} \mu\left(B(x_0,\rho(x_0))\right)
		\lesssim R^{d-2}.  
	\end{align*}
	
	Now, if $R> \rho(x_0)$, let $j_0 \in \mathbb{N}$ such that $2^{j_0-1}\rho(x_0)< R\leq 2^{j_0}\rho(x_0)$ and use \eqref{eq: prop 2 mu} $j_0$ times to get
	\[\mu(B(x_0,2^{j_0}r))\leq D_\mu^{j_0}\mu(B(x_0,r))+D_\mu^{j_0}r^{d-2}2^{(j_0-1)(d-2)}.\]
	
	Then, again by \eqref{eq: prop 1 mu} and \eqref{eq: mu equiv critical balls} 
	\begin{align*}
		\mu\left(B(x_0,R)\right)&\lesssim \left(\frac{R}{2^{j_0}\rho(x_0)}\right)^{d-2+\delta_\mu} \mu\left(B(x_0,2^{j_0}\rho(x_0))\right)\\
		& \lesssim \left(\frac{R}{2^{j_0}\rho(x_0)}\right)^{d-2+\delta_\mu}\left(D_\mu^{j_0} \mu\left(B(x_0,\rho(x_0))\right)+D_\mu^{j_0}\rho(x_0)^{d-2}2^{(j_0-1)(d-2)}\right)\\
		&\leq C  R^{d-2}\left(\frac{R}{\rho(x_0)}\right)^{\delta_\mu}\left(\frac{D_\mu}{2^{\delta_\mu}}\right)^{j_0}.
	\end{align*}
	
	Finally, since 
	$\log_2\left(\frac{R}{\rho(x_0)}\right)\leq j_0\leq 1+\log_2\left(\frac{R}{\rho(x_0)}\right)$, \begin{align*}
		D_\mu^{j_0}2^{-j_0\delta_\mu} &\leq D_\mu^{1+\log_2\left(\frac{R}{\rho(x_0)}\right)} 2^{-\delta_\mu\log_2\left(\frac{R}{\rho(x_0)}\right)}\\
		&\leq D_\mu D_\mu^{\log_2\left(1+\frac{R}{\rho(x_0)}\right)} \left(\frac{R}{\rho(x_0)}\right)^{-\delta_\mu}\\ 
		&\lesssim \left(1+\frac{R}{\rho(x_0)}\right)^{\log_2 D_\mu}\left(\frac{R}{\rho(x_0)}\right)^{-\delta_\mu}.
	\end{align*}
	Therefore, in both cases we obtain the desired estimate.
\end{proof}

\subsection{Riesz transforms}

We consider the Riesz transform $R_\mu=\nabla \mathcal{L}_\mu^{-1/2}$ and its adjoint ${R^*_\mu=\mathcal{L}_\mu^{-1/2} \nabla }$  as in \cite{Shen99}.

The Riesz transform $R_\mu$ can be expressed, for $f\in C_c^\infty(\mathbb{R}^d)$ and   $x\notin \supp(f)$, as 
\begin{align*}
	R_\mu f(x) 
	& = \frac{1}{\pi}\int_{0}^{\infty} \lambda ^{-\frac{1}{2}} \nabla(\mathcal{L}_\mu +\lambda)^{-1}f(x) d\lambda\\
	& = \frac{1}{\pi} \int_{0}^{\infty} \lambda^{-\frac{1}{2}} \int_{\mathbb{R}^d} \nabla_1 \Gamma_{\mu+ \lambda }(x,y)f(y)dy \,d\lambda.
\end{align*}
By Fubini's theorem,
\begin{align*}
	R_\mu f(x) & = \int_{\mathbb{R}^d} \frac{1}{\pi} \int_{0}^{\infty}\lambda^{-\frac{1}{2}}  \nabla_1 \Gamma_{\mu+ \lambda }(x,y)d\lambda f(y)dy = \int_{\mathbb{R}^d} K_\mu(x,y) f(y)dy,
\end{align*}
where $K_\mu $ is the singular kernel of $R_\mu$ given by
\begin{equation*}
	K_\mu(x,y)= \frac{1}{\pi} \int_{0}^{\infty}\lambda^{-\frac{1}{2}}  \nabla_1 \Gamma_{\mu+ \lambda }(x,y)d\lambda. 
\end{equation*}

The adjoint $R_\mu^*$ can be written as
\[ R^*_\mu f(x)=\int_{\mathbb{R}^d}K^*_\mu(x,y)f(y)dy= \int_{\mathbb{R}^d}-K_\mu (y,x)f(y)dy.\]

By \cite[Theorem~7.18]{Shen99} we know that $R_\mu$ is a Calder\'on--Zygmund operator when $\delta_\mu>1$, so it is bounded on $L^p(\mathbb{R}^d)$ for every $1<p<\infty$. By duality, $R_\mu^*$ is also bounded on $L^p(\mathbb{R}^d)$ for every $1<p<\infty$. There also is proved condition \eqref{eq: smooth pointwise}.  
Moreover, we have that  (see \cite[Theorem~0.17]{Shen99}) 
\begin{equation*}
	|	K_\mu(x,y)|\leq C \frac{e^{- \epsilon d_\mu(x,y)}}{|x-y|^{d}},
\end{equation*}
for some $\epsilon>0$, and the same estimate hold for $|K^*_\mu(x,y)|$ (see also \cite[Lemma 4.3]{Bailey21}).
Then, by Lemma~\ref{lem: drho lower bound} we obtain the following classification for the Riesz transform and its adjoint.  

\begin{prop}\label{prop: Riesz infty-delta}
	Let  $\mu$ be a measure verifying \eqref{eq: prop 1 mu} and \eqref{eq: prop 2 mu} with  ${\delta_\mu>1}$. Then, $R_\mu$ and $R_\mu^*$  are both exponential Schr\"odinger--Calder\'on--Zygmund operators of $(\infty,\delta)$ type as given in Definition \ref{def: infty-delta type} with constants $c=\frac{\epsilon}{2D_1}$ and $m=\frac{1}{k_0+1}$. 
\end{prop}

As a consequence of Proposition \ref{prop: Riesz infty-delta} above,  Proposition \ref{prop: endpoint infty-delta} and Theorem \ref{thm: extrapolation H weights} we obtain the following boundedness result for the case~${\delta_\mu>1}$.  

\begin{thm}
	Let  $\mu$ be a measure verifying \eqref{eq: prop 1 mu} and \eqref{eq: prop 2 mu} with  $\delta_\mu>1$. Then,
	if $m=\frac{1}{k_0+1}$, for any $c<\frac{\epsilon}{2D_1} {2^{-m}}$   and  $w$ such that $w^{-1} \in H_{1, c}^{\rho, m}$, we have that	 $R_\mu$ and $R_\mu^*$  are both bounded from $L^\infty (w)$ into $\BMO(w)$.  Moreover, both are bounded on $L^p(w)$, for every $1<p<\infty$ and {$w\in H_{p,c^*}^{\rho,m^*}$ with $m^*=\frac{m}{k_0+1}$ and $c^*<c(4C_0)^{-2m^*}$}.
\end{thm}	

When $0<\delta_\mu<1$ we consider the adjoint operators $R_\mu^*$ and the estimates given in \cite[Lemma~4.3]{Bailey21}, which establishes the existence of constants $C,\epsilon>0$ such that 
\begin{equation}\label{eq: adjoint Riesz kernel size}
	|K^*_\mu(x,y)| \leq C \frac{e^{-\epsilon d_\mu(x,y)}}{|x-y|^{d-1}}\left(\int_{B \left(y,\frac{|y-x|}{2}\right)} \frac{d\mu(z)}{|z-y|^{d-1}} + \frac{1}{|x-y|}\right), 
\end{equation}	
for all $x,y\in \mathbb{R}^d$ with $x\neq y$.

We will now determine the type of operator that $R_\mu^*$ is, for $0<\delta_\mu<1$, and the corresponding boundedness results.

\begin{prop}\label{prop: adjoint Riesz s-delta}
	Let  $\mu$ be a measure verifying \eqref{eq: prop 1 mu} and \eqref{eq: prop 2 mu} with  $0<\delta_\mu<1$. Then $R_\mu^*$ is an exponential Schr\"odinger--Calder\'on--Zygmund operator of $(s,\delta)$ type  for some $0<\delta<1$ and  every $1<s<\frac{2-\delta_\mu}{1-\delta_\mu}$ as given in Definition \ref{def: s-delta type} with constants $c=\frac{\epsilon}{4D_1}$ and $m=\frac{1}{k_0+1}$.
\end{prop}

\begin{proof} Condition \ref{itm: weak-type T} of Definition~\ref{def: s-delta type} follows from the proof of \cite[Theorem~7.1]{Shen99} for every  $s'>2-\delta_\mu$, i.e., for every $1<s<\frac{2-\delta_\mu}{1-\delta_\mu}$.
	
	We now prove condition \eqref{eq: size Hormander}. Using \cite[Lemma~7.9]{Shen99} with  $1<s<\frac{2-\delta_\mu}{1-\delta_\mu}$ and Lemma \ref{lem: mu extra decay},  we have that, whenever $|x-x_0|<R/2$,
	\begin{align}\label{eq: G estimate}
		\left(\int_{R<|x_0-y|\leq 2R}\left(\int_{B \left(y,\frac{|y-x|}{2}\right)}\frac{d\mu(z)}{|z-y|^{d-1}}\right)^s dy \right)^{1/s}&\leq 
		\left(\int_{B(x,4R)}\left(\int_{B(x,4R)}\frac{d\mu(z)}{|z-y|^{d-1}}\right)^s dy \right)^{1/s}\nonumber \\
		&\leq C
		\frac{\mu\left(B(x,12R)\right)}{R^{d(1-(1/s))-1}} 
		\nonumber \\
		&
		\leq C R^{d/s-1} \left(1+\frac{R}{\rho(x)}\right)^{\log_2 D_\mu}.
	\end{align}
	Then, by Lemma \ref{lem: drho lower bound} and \eqref{eq: adjoint Riesz kernel size},
	\begin{align*}
		\left(\frac{1}{R^d}\int_{R<|x_0-y|\leq 2R}|K^*_\mu(x,y)|^s dy \right)^{1/s}& \lesssim \frac{1}{R^{d}}e^{-\frac{\epsilon}{2D_1}  \left(1+\frac{R}{\rho(x)}\right)^{\frac{1}{k_0+1}}}\left(1+\frac{R}{\rho(x)}\right)^{\log_2 D_\mu}\\
		&\lesssim \frac{1}{R^{d}}e^{-c  \left(1+\frac{R}{\rho(x)}\right)^{m}}. 	
	\end{align*}

	To prove condition \eqref{eq: smooth Hormander} we follow the ideas given in the proof of \cite[Theorem~7.18]{Shen99}. There exists $\delta_1 \in (0,1)$ such that for ${x_0 \in B\left(x,\frac{|x-y|}{8}\right)}$,
	\begin{align*}
		| \nabla_1 &\Gamma_{\mu+ \lambda}  (y,x) -  \nabla_1 \Gamma_{\mu+ \lambda }(y,x_0)|\\
		& \lesssim
		\left(\frac{|x-x_0|}{|x-y|}\right)^{\delta_1} \sup_{w\in B(x,|x-y|/2)} |  \nabla_1 \Gamma_{\mu+ \lambda }(y,w)|
		\\
		& \lesssim
		\left(\frac{|x-x_0|}{|x-y|}\right)^{\delta_1} \sup_{w\in B(x,|x-y|/2)} \left[e^{-\epsilon_3 d_{\mu+\lambda}(y,w)}\frac{1}{|y-w|^{d-2}} \left(\int_{B\left(y,\frac{|w-y|}{2}\right)}\frac{d\mu(z)}{|{z-y}|^{d-1}}+ \frac{1}{|w-y|}\right)\right]
		\\
		& \lesssim
		\left(\frac{|x-x_0|}{|x-y|}\right)^{\delta_1} \sup_{w\in B(x,|x-y|/2)} \left[e^{-\frac{\epsilon_3}{2} |y-w|}\frac{1}{|y-w|^{d-2}} \left(\int_{B\left(y,\frac{|w-y|}{2}\right)}\frac{d\mu(z)}{|{z-y}|^{d-1}}+ \frac{1}{|w-y|}\right)\right],
	\end{align*}
	where we have used \eqref{eq: grad fundamental solution bound} and the inequality 
	\begin{equation}\label{eq: d_mu+lambda ineq}
		d_{\mu+\lambda}(x,y)\geq \frac12\left(d_\mu(x,y)+d_\lambda(x,y)\right)\geq \frac12d_\mu(x,y)+ \frac{\sqrt{\lambda}}{2}|x-y|
	\end{equation}
	given in \cite[(21)]{Bailey21}. 
	
	By taking the supremum on $w$ we obtain 
	\begin{align*}
		&|K^*_\mu(x,y)-K^*_\mu(x_0,y)|\\
		&\lesssim \left(\frac{|x-x_0|}{|x-y|}\right)^{\delta_1} \frac{1}{|y-x|^{d-2}}\left(\int_{B(y,|y-x|)}\frac{d\mu(z)}{|z-y|^{d-1}}+ \frac{1}{|y-x|}\right) \int_0^{\infty}  \frac{e^{-\frac{\epsilon_3}{4} \lambda^{1/2}}}{\lambda^{1/2}} d\lambda\\
		& \lesssim \left(\frac{|x-x_0|}{|x-y|}\right)^{\delta_1} \frac{1}{|y-x|^{d-2}}\left(\int_{B(y,|y-x|)}\frac{d\mu(z)}{|z-y|^{d-1}}+ \frac{1}{|y-x|}\right). 
	\end{align*}
	
	Finally, proceeding as in the estimate of the kernel size and using \eqref{eq: G estimate}, we obtain \eqref{eq: smooth Hormander}, completing the proof of the proposition.
\end{proof}

By Propositions \ref{prop: adjoint Riesz s-delta} and \ref{prop: endpoint s-delta}, and Theorem \ref{thm: extrapolation H weights}, we deduce the boundedness result for the case ${0<\delta_\mu<1}$.  

\begin{thm}
	Let  $\mu$ be a measure verifying \eqref{eq: prop 1 mu} and \eqref{eq: prop 2 mu} with  $0<\delta_\mu<1$. For $1<s<\frac{2-\delta_\mu}{1-\delta_\mu}$, $m=\frac{1}{k_0+1}$, $c<s' \frac{\epsilon}{4D_1}{2^{-m}}$   and  $w$ such that ${w^{-s'} \in H_{1, c}^{\rho, m}}$ we have that  $R_\mu^*$  is bounded from $L^\infty (w)$ into $\BMO(w)$. Moreover it is bounded on $L^p(w)$, for every $s'<p<\infty$ and {$w\in H_{p/s',c^*}^{\rho,m^*}$ with $m^*=\frac{m}{k_0+1}$ and $c^*<s'c(4C_0)^{-2m^*}$}.
\end{thm}

\subsection{Laplace transform type multipliers}

Let $\phi\in L^\infty(\mathbb{R}_+)$, and let 
\[m_\phi(\lambda)=\lambda \int_0^\infty e^{-\lambda t}\phi(t)\, dt, \quad \lambda>0.\]
By means of the Spectral Theorem, we can define the Laplace transform type multipliers $m_\phi(\mathcal{L}_\mu)$, which are bounded on $L^2(\mathbb{R}^d)$ for any $\phi\in L^\infty(\mathbb{R}_+)$ (see \cite[Corollary~3,~p.~121]{SteinTopics}).

Special cases of multipliers of Laplace transform type are the imaginary powers $\mathcal{L}_\mu^{i\gamma}$ that arise when $m_\phi(\lambda)=\lambda^{i\gamma}$ for $\lambda>0$ and $\gamma\in \mathbb{R}$; that is, when $\phi(t)=t^{-i\gamma}/\Gamma(1-i\gamma)$.

If we consider the heat semigroup associated with $\mathcal{L}_\mu$,
\[\mathcal W_tf(x):=e^{-t\mathcal{L}_\mu}f(x)=\int_{\mathbb{R}^d} \mathcal W_t(x,y) f(y)dy, \quad f \in L^2(\mathbb{R}^d),\, x \in \mathbb{R}^d, t>0, \]
we can write
\[m_\phi(\mathcal{L}_\mu)f(x)= \int_{0}^{\infty} \phi(t) \mathcal{L} e^{-t\mathcal{L}_\mu}f(x)dt=\int_{0}^{\infty}\phi(t) \partial_t \mathcal W_tf(x) dt, \qquad x \in \mathbb{R}^d.\]
Its kernel can be expressed as
\[\mathcal{M}_\phi(x,y)= \int_{0}^{\infty} \phi(t) \partial_t \mathcal{W}_t(x,y) dt, \quad x,y\in \mathbb{R}^d.\] 

In order to classify the operator  $m_\phi(\mathcal{L}_\mu)$, we will need some estimates for the derivative of the heat kernel $\mathcal W_t$, already proved by Wu and Yan in \cite{WuYan16} (for $d\mu(x)=V(x) dx$, see \cite[Proposition 4]{DGMTZ}). We state them below. 

\begin{lem}[{\cite[Lemma~2.3 and Lemma~3.8]{WuYan16}}]\label{lem: Q_t kernel estimates} \leavevmode
	\begin{enumerate}
		\item \label{itm: t partial deriv W_t} 
		Let $k_0$ as in Definition~\ref{def: critical radius function}.
		There exist constants $c_0$ and $C$ such that \[\left|t \partial_t\mathcal{W}_t(x,y)\right| \leq C t^{-d/2} {e^{-\frac{|x-y|^2}{4t}}}\exp\left(-c_0\left(1+\frac{\max\{|x-y|, \sqrt{t/2}\}}{\rho(x)}\right)^{\frac{1}{k_0+1}}\right).\]
		\item \label{itm: t difference partial deriv W_t} For every $ 0<\delta<\min\{1,\delta_\mu\}$ and $N\in \mathbb{N}$, there exist constants $c$ and $C_N$ such that for any $|h|\leq \sqrt{t}$, 
		\[|t \partial_t \mathcal{W}_t(x+h,y)-t \partial_t \mathcal{W}_t(x,y)| \leq   \left(\frac{|h|}{\sqrt{t}}\right)^{\delta} t^{-d/2}{e^{-\frac{|x-y|^2}{4t}}}\frac{C_N}{\left(1+\frac{\sqrt{t}}{\rho(x)}+\frac{\sqrt{t}}{\rho(y)}\right)^N},\]
		with $x,y\in \mathbb{R}^d$ and $t>0$.
	\end{enumerate}
\end{lem}

We are now in position to establish the following classification for   $m_\phi(\mathcal{L}_\mu)$.

\begin{prop}
	Let $\phi$ be as above. The operator $m_\phi(\mathcal{L}_\mu)$  is an exponential Schr\"odinger--Calder\'on--Zygmund operator of $(\infty,\delta)$ type as given in Definition~\ref{def: infty-delta type} with constants $c=c_0$ (the constant in Lemma~\ref{lem: Q_t kernel estimates}\ref{itm: t partial deriv W_t}) and $m=\frac{1}{k_0+1}$. 
\end{prop}

\begin{proof}
	We shall first prove \eqref{eq: size pointwise}. For any $x,y\in \mathbb{R}^d$,	by Lemma~\ref{lem: Q_t kernel estimates}\ref{itm: t partial deriv W_t},
	\begin{align*}
		\left|\mathcal{M}_\phi(x,y)\right|
		& 
		\leq \|\phi\|_\infty \int_{0}^{\infty}  |\partial_t \mathcal{W}_t(x,y)| dt\\
		&\leq C \exp\left(-c_0\left(1+\frac{|x-y|}{\rho(x)}\right)^{\frac{1}{k_0+1}}\right)\int_{0}^{\infty} t^{-d/2} {e^{-\frac{|x-y|^2}{4t}}} \frac{dt}{t}.
	\end{align*}
	By performing the change of variables $u=\frac{|x-y|^2}{t}$, we get
	\begin{align*}
		\left|\mathcal{M}_\phi(x,y)\right|&\leq  C \exp\left(-c_0\left(1+\frac{|x-y|}{\rho(x)}\right)^{\frac{1}{k_0+1}}\right)\frac{1}{|x-y|^d} \int_0^\infty u^{d/2}e^{-u/4} \frac{du}{u}\\
		&\leq \frac{C}{|x-y|^d} \exp\left(-c_0\left(1+\frac{|x-y|}{\rho(x)}\right)^{\frac{1}{k_0+1}}\right),
	\end{align*}
	and \eqref{eq: size pointwise} holds with $c_1=c_0$ and $m=\frac{1}{k_0+1}$. 
	
	For the smoothness condition \eqref{eq: smooth pointwise}, we consider $|x-y|>2|x- x_0|$ and $0<\delta<\min\{1, \delta_\mu\}$. We have that
	\begin{align}\label{eq: kernel Laplace difference}
		\nonumber |\mathcal{M}_\phi(x,y)-\mathcal{M}_\phi(x_0,y)|&\leq \|\phi\|_\infty \int_0^\infty |\partial_t \mathcal{W}_t(x,y)- \partial_t \mathcal{W}_t(x_0,y)| dt\\
		&\nonumber \sim\int_0^{|x-x_0|^2} |t\partial_t \mathcal{W}_t(x,y)- t\partial_t \mathcal{W}_t(x_0,y)| \frac{dt}{t}\\
		&\quad +\int_{|x-x_0|^2}^\infty |t\partial_t \mathcal{W}_t(x,y)- t\partial_t \mathcal{W}_t(x_0,y)| \frac{dt}{t}.
	\end{align}
	
	For the second term,  as  $|x-x_0|\leq \sqrt{t}$, by Lemma~\ref{lem: Q_t kernel estimates}\ref{itm: t difference partial deriv W_t} with $h=x-x_0$ it follows that
	\begin{align*}
		|t\partial_t \mathcal{W}_t(x,y)&- t\partial_t \mathcal{W}_t(x_0,y)|\\
		&\leq C_1 \left(\frac{|x-x_0|}{\sqrt{t}}\right)^{\delta} t^{-d/2}\exp\left(-\frac{|x-y|^2}{4t}\right)\\
		&\leq C_\delta \left(\frac{|x-x_0|}{\sqrt{t}}\right)^{\delta}\left(\frac{|x-y|}{\sqrt{t}}\right)^{-\delta} t^{-d/2}\exp\left(-\frac{|x-y|^2}{8t}\right)\\
		&= C_\delta \left(\frac{|x-x_0|}{|x-y|}\right)^{\delta} t^{-d/2}\exp\left(-\frac{|x-y|^2}{8t}\right).
	\end{align*} 
	Then
	\begin{align*}
		\int_{|x-x_0|^2}^\infty &|t\partial_t \mathcal{W}_t(x,y)- t\partial_t \mathcal{W}_t(x_0,y)| \frac{dt}{t}\\
		&\leq C\left(\frac{|x-x_0|}{|x-y|}\right)^{\delta}\int_{|x-x_0|^2}^\infty t^{-d/2} \exp\left(-\frac{|x-y|^2}{8t}\right)\frac{dt}{t}\\
		&\leq C\frac{1}{|x-y|^d} \left(\frac{|x-x_0|}{|x-y|}\right)^\delta \int_0^\infty u^{d/2-1} e^{-u/8} du\\
		&\leq C\frac{1}{|x-y|^d} \left(\frac{|x-x_0|}{|x-y|}\right)^\delta.
	\end{align*}

	For $|x-x_0|>\sqrt{t}$, we use Lemma~\ref{lem: Q_t kernel estimates}\ref{itm: t partial deriv W_t} to get, for $\delta$ as before, that
	\begin{align*}
		|t\partial_t \mathcal{W}_t(x,y)|&\leq C t^{-d/2} \exp\left(-\frac{|x-y|^2}{4t}\right)\\
		&\leq C \left(\frac{|x-x_0|}{\sqrt{t}}\right)^{\delta} t^{-d/2} \exp\left(-\frac{|x-y|^2}{4t}\right)\\
		&\leq 
		C_\delta \left(\frac{|x-x_0|}{|x-y|}\right)^{\delta} t^{-d/2} \exp\left(-\frac{|x-y|^2}{8t}\right).
	\end{align*}

	And, since $|x-y|\sim |x_0-y|$, we also obtain the same bound for $|t\partial_t \mathcal{W}_t(x_0,y)|$ when $|x-x_0|>\sqrt{t}$. Therefore, the first term in \eqref{eq: kernel Laplace difference} is bounded as desired. 
	
	Finally, as previously mentioned, $m_\phi(\mathcal{L}_\mu)$ is bounded on $L^2(\mathbb R^d)$. This, together with the conditions above, allows us to conclude that it is a Calderón-Zygmund type operator. Therefore, $m_\phi(\mathcal{L}_\mu)$ is bounded on $L^p(\mathbb{R}^d)$ for all $p>1$, and the proposition is thus proved.
\end{proof}

As an immediate consequence we obtain the following.
\begin{thm}
	Let  $\mu$ be a measure verifying \eqref{eq: prop 1 mu} and \eqref{eq: prop 2 mu}.
	Then, if ${m=\frac{1}{k_0+1}}$,
	for any $c<c_0{2^{-m}}$  and  $w$ such that $w^{-1} \in H_{1, c}^{\rho, m}$, we have that	 $m_\phi(\mathcal{L}_\mu)$ is bounded from $L^\infty (w)$ into $\BMO(w)$.  Moreover, it is  bounded on $L^p(w)$, for every $1<p<\infty$ and  {$w\in H_{p,c^*}^{\rho,m^*}$ with $m^*=\frac{m}{k_0+1}$ and $c^*<c(4C_0)^{-2m^*}$}.
\end{thm}

\subsection{Operators associated to a potential.}

In the particular case where the measure $\mu$ is given by $d\mu(x)= V(x) dx$, with $V$ a nonnegative function belonging to the classical reverse H\"older class $RH_q$ for $q>\frac d2$, we will be able to study the operators $T_j= (- \Delta +V)^{-j/2} V^{j/2}$ for $j=1,2$. Actually, they can be found in \cite{BCH13Extrapol}, where weighted estimates were proved for weights in the class  $A_p^{\rho}$, where we understand by $\rho$ the critical radius function associated to the potential $V$. 

Condition \ref{itm: weak-type T} of Definition \ref{def: s-delta type} holds for $T_j$, $j=1,2$  when $s\geq2q$ and $s\geq q$, respectively (see \cite[Theorems~5.10 and 3.1]{Shen95}) and condition  \eqref{eq: smooth Hormander} is contained in the proof of \cite[Theorem~1]{GLP08} when $s=2q$ and $s=q$ for $T_1$ and $T_2$, respectively, and some $\delta>0$.

Now, from the results in \S\ref{sec: proofs main results}, we shall obtain estimates for weights in ${H_{p,c}^{\rho,m}}$, for appropriated parameters. In order to do so, we will see that the size condition obtained in \cite[Theorem~8]{BCH13Extrapol} can be improved in order to have \eqref{eq: size Hormander} for the kernels of $T_j$, $j=1,2$.

For each $j=1,2$, we have 
\begin{align*}
	T_j f(x) &= (-\Delta + V)^{-j/2}V^{j/2}f(x)\\
	& = \frac{1}{\pi}\int_{0}^{\infty} \lambda ^{-j/2}\int_{\mathbb{R}^d}   \Gamma_{V+ \lambda }(x,y) V(y)^{j/2} f(y) dy d\lambda
	\\
	&= \int_{\mathbb{R}^d} K_j(x,y)f(y) dy,
\end{align*}
and by the  estimates for the fundamental solution of  $-\Delta +V$  given in \eqref{eq: fundamental solution bounds}, \eqref{eq: d_mu+lambda ineq} and Lemma \ref{lem: drho lower bound}, there exists $\epsilon_j>0$ such that  
\begin{equation}\label{eq: K_j kernel bound}
	|K_j(x,y)|\leq C  \frac{\exp\left(-\frac{\epsilon_j}{2D_1} \left(1+\frac{|x-y|}{\rho(x)}\right)^{\frac{1}{k_0+1}}\right)}{|x-y|^{d-j}} V(y)^{j/2}, \quad j=1,2. \end{equation}

Then, condition \eqref{eq: size Hormander} follows in the same way as in the proof of \cite[Theorem 8]{BCH13Extrapol} considering $s=2q$ and $s=q$ for the cases $j=1$ and $j=2$, respectively. 

Therefore, we get that both $T_j$ are exponential Schr\"odinger--Calder\'on--Zygmund operators of $(s, \delta)$ type for $s$ as above and some  $\delta>0$. 

As a consequence of Proposition \ref{prop: endpoint s-delta} and Theorem \ref{thm: extrapolation H weights} we obtain the following  result for each $T_j$.

\begin{thm}
	Let $V$ be as above with   $q>\frac{d}2$ , and let $\epsilon_j>0$ be as in \eqref{eq: K_j kernel bound}, $j=1,2$. Then, if  $m=\frac{1}{k_0+1}$, for any $c_j<\frac{\epsilon_j}{2D_1}{2^{-m}}$ and  $w$ such that $w^{-s'} \in H_{1, c_j}^{\rho, m}$, with $s=2q$ when $j=1$ and   $s=q$ when $j=2$,  we have that	$T_j$  is bounded from $L^\infty (w)$ into $\BMO(w)$. Furthermore, $T_j$ is also bounded on $L^p(w)$ for every $s'<p<\infty$ and {$w\in H_{p/s',c_j^*}^{\rho,m^*}$ with $m^*=\frac{m}{k_0+1}$ and $c_j^*<s'c_j(4C_0)^{-2m^*}$, for $j=1,2$.}
\end{thm}

\bibliographystyle{acm}

\end{document}